\documentclass[11pt]{article}

\usepackage[margin=1in]{geometry} 
\usepackage{amsmath,amsthm,amssymb}
\usepackage{nicematrix}
\usepackage{verbatim}
\usepackage[normalem]{ulem}
\usepackage{algorithm}
\usepackage{algpseudocode}
\usepackage{authblk}
\usepackage[T1]{fontenc}
\usepackage{arydshln}

\newcommand{\N}{\mathbb{N}}
\newcommand{\Z}{\mathbb{Z}}
\newcommand{\Q}{\mathbb{Q}}
\newcommand{\R}{\mathbb{R}}

\newtheorem{theorem}{Theorem}
\newtheorem{proposition}[theorem]{Proposition}
\newtheorem{claim}[theorem]{Claim}
\newtheorem{lemma}[theorem]{Lemma}
\newtheorem{corollary}[theorem]{Corollary}

\newtheorem{definition}[theorem]{Definition}
\newtheorem{example}[theorem]{Example}
\newtheorem{remark}[theorem]{Remark}
\newtheorem{observation}{Observation}

\usepackage{hyperref}

\begin{document}





\title{Scarf's algorithm and stable marriages} 
\author{Yuri Faenza\footnote{IEOR Department, Columbia University, yf2414@columbia.edu}}
\author{Chengyue He\footnote{IEOR Department, Columbia University, ch3480@columbia.edu}} 
\author{Jay Sethuraman\footnote{IEOR Department, Columbia University, jay@ieor.columbia.edu}} 
 
\affil{Columbia University} 
\maketitle



\begin{abstract}
    Scarf's algorithm gives a pivoting procedure to find a special vertex---a  \emph{dominating} vertex---in down-monotone polytopes. This paper studies the behavior of Scarf's algorithm when employed to find stable matchings in bipartite graphs. First, it proves that Scarf's algorithm can be implemented to run in polynomial time, showing the first positive result on its runtime in significant settings. Second, it shows an infinite family of instances where, no matter the pivoting rule and runtime, Scarf's algorithm outputs a matching from an exponentially small subset of all stable matchings, thus showing a structural weakness of the approach. 
\end{abstract}

\maketitle

%


%
%
%

\section{Introduction}

The theory of stable matchings has been studied for decades by the algorithms and operations research community. This effort has led to a variety of algorithms, which often give complementary approaches to the same problem. For instance, when we are given weights on the edges, a stable matching of maximum total weight in a bipartite (marriage) instance can be found using a combinatorial algorithm~\cite{irving1987efficient}, any linear programming solver~\cite{roth1993stable,rothblum1992characterization,vate1989linear}, or an interplay of the two~\cite{faenza2022affinely}. The problem of finding a stable matching in a marriage instance can then also be solved via multiple algorithms, including Gale and Shapley's Deferred Acceptance algorithm~\cite{gale1962college}, mechanisms using compensation chains~\cite{dworczak2021deferred}, and more (see~\cite{manlove2013algorithmics} for extensive references).

Scarf's lemma~\cite{scarf1967core} provides yet another algorithm for finding a stable matching in a marriage instance, with a number of 
distinct features that call for a deeper understanding. First, it is geometric in nature, while most other algorithms are combinatorial. Secondly, it applies under more general conditions, that go well beyond the classical marriage setting by Gale and Shapley. 
For instance, it has been used to design heuristics or (not necessarily polynomial-time) exact algorithms for matching markets with complex side constraints for which no alternative algorithms are known; examples of such side constraints are those arising from the presence of couples or budgets or the need to meet a proportionality requirement~\cite{biro2016matching, nguyen2021stability, nguyen2018near,nguyen2019stable}. Scarf's lemma has also been employed to show the existence of objects arising in the theory of graphs, matroids, posets, and games~\cite{aharoni2003lemma,aharoni1995fractional,biro2016matching}. Its connections with Sperner's lemma~\cite{kiraly2009note} and Nash equilibria~\cite{Scarf} have also been studied.

\smallskip

In essence, Scarf's lemma guarantees the existence of a vertex of a down-monotone polytope that is \emph{dominating} (see Section~\ref{sec:scarf-lemma} for definitions). The original proof by Scarf is algorithmic, as it is based on a pivoting rule that, starting from a certain vertex of the polytope, is guaranteed to terminate at a dominating vertex. When applied to the bipartite matching polytope, it guarantees the existence of a stable matching in a marriage instance. When applied to the fractional matching polytope, and combined with its half-integrality, it guarantees the existence of a stable partition in a roommate instance, a result originally obtained by Tan~\cite{tan1991stable} via a combinatorial argument. As discussed earlier, however, Scarf's lemma applies much more generally: for instance, when applied to a fractional hypergraph matching polytope, it establishes the existence of a fractional stable matching in a hypergraph~\cite{aharoni1995fractional}. Many allocation problems in which integral solutions may not exist can be modelled as the problem of finding a fractional stable matching in a hypergraph~\cite{biro2016matching}, which is then rounded to an integer, almost-feasible solution~\cite{nguyen2018near}.

The generality of Scarf's result comes however at a computational price: the problem of finding a dominating vertex of a down-monotone polytope is PPAD-Complete~\cite{kintali2013reducibility}. In particular, it remains PPAD-Complete on the hypergraph matching polytope~\cite{kintali2013reducibility}, even in quite restricted settings~\cite{csaji2021complexity,ishizuka2018complexity}.

\smallskip

These negative results frustrate the search for a proof of polynomial-time convergence of Scarf's algorithm in its most general terms. However, the broad applicability of Scarf's lemma calls for a more fine-grained analysis of the algorithm. To the best of our knowledge, in no relevant case has the polynomial-time convergence of Scarf's algorithm been established so far. This is in stark contrast with the many positive results known for other important geometric algorithms in combinatorics and combinatorial optimization, including the cutting plane method~\cite{chandrasekaran2016cutting} and pivoting procedures such as the simplex method~\cite{black2021simplex,orlin1997polynomial,tarjan1997dynamic}. 

The goal of this paper is to shed some light on the strengths and limits of Scarf's algorithm, by focusing on its application to the marriage model. The restriction to this model is motivated by multiple reasons. First, the marriage model is arguably one of the most relevant settings where Scarf's lemma holds true, and  applications of Scarf's algorithm in market design rely on extensions of the marriage model~\cite{nguyen2018near,nguyen2019stable,nguyen2021stability}. For some of those applications, we do not know if Scarf's algorithm can be implemented to run in polynomial time, and understanding their common special case is a natural first step. Second, computational experiments have shown that Scarf's algorithm can be used as a heuristic in some of those markets~\cite{biro2016matching}, hence understanding which underlying structure implies fast running time is an intriguing and important question. In particular, Biro and Fleiner~\cite{biro2016fractional} pose many open questions on the features of Scarf's algorithm when employed to find stable matchings. Some of these open questions are answered in this paper, see Section~\ref{sec:contribution} and Section~\ref{sec:related}. Last, matching problems have often proved to be at the ``proper level of difficulty''~\cite{lovasz2009matching} for developing tractable, yet non-trivial, theories and algorithms. 





Before delving into the details, we formally introduce Scarf's lemma and its associated algorithm in the next section and review them more in detail in Section~\ref{sec:ReviewScarf}. Our results are presented formally in Section~\ref{sec:contribution}, and related literature is discussed in Section~\ref{sec:related}. An overview of the techniques used to show our results is given in Section~\ref{sec:overview}, while full details are carried out in Sections~\ref{sec:algo-bipartite},~\ref{sec:perturbation},  and~\ref{sec:not-all-stable}. 

\addtocontents{toc}{\setcounter{tocdepth}{-10}} 

\subsection{Scarf's Lemma and Applications}\label{sec:scarf-lemma}


Consider a \emph{down-monotone} polytope $P\subset  \mathbb{R}^{(n+m)}$ in standard form, i.e.,  \begin{equation}\label{eq:Ax-leq-b}
P=\{x \in \mathbb{R}^{n+m}_{\ge 0} : Ax = b\}, \hbox{ with } A=(I|A'),\end{equation} 
where $I$ is the $n\times n$ identity matrix, $A \in \mathbb{Q}^{n \times (n+m)}_{\geq 0}$, $A'=(a'_{i,j})$, and $b\in\Q^n_+$. We say that $A$ as above is in \emph{standard form}. We call a basis (in the classical linear algebra sense) $B$ for $A$ that is feasible for~\eqref{eq:Ax-leq-b} an \emph{$(A,b)$ basis}. The definition depends on $b$ to emphasize feasibility. 

A matrix $C=(c_{i,j}) \in \Z^{n\times (n+m)}$ is an \emph{ordinal matrix} if it has distinct entries that satisfy $c_{i,i} < c_{i,k} < c_{i,j}$  for any $i\neq j \in [n]$, $k \in [n+m]\setminus [n]$, where we let $[n]=\{1,\dots,n\}$. As we see later, the only relevant information contained in $C$ is the relative order of entries in each row of $C$, so we assume w.l.o.g.~$c_{ij}=O(n+m)$ for all $i,j$.

Consider a set $D=\{j_1,\dots,j_n\}$ of $n$ columns of $C$. 
For every row $i \in [n]$, define the minimum element
\begin{equation}\label{eq:utilityvector}
    u_i=\min_{k\in[n]}c_{i,j_k}.
\end{equation}

Fix $j \in D$. By definition, $u_i\le c_{i,j}$ for any row index $i$. We succinctly write this set of relations by $u\le c_{j}$.
The set $D$ is called an \emph{ordinal basis} of $C$ if for every column $h\in[m]$, there is at least one $i \in [n]$ such that $u_i\ge c_{ih}$. The associated vector $u\in\Q^n$ is called the \emph{utility vector} of this ordinal basis. These concepts are illustrated in the example below.

\begin{example}\label{ex:utility-vector}
The matrix, $C$, below is an ordinal matrix and $D=\{\mathit{4},\mathit{5},\mathit{6}\}$ is an ordinal basis with $u=(1,1,1)^T$.
$$C=
\begin{pNiceArray}{ccc|ccc}[first-row]
    \mathit{1} & \mathit{2} & \mathit{3} & \mathit{4} & \mathit{5} & \mathit{6} \\
    0 & 5 & 4 & 2 & 3 & 1 \\
    5 & 0 & 4 & 1 & 2 & 3 \\
    5 & 4 & 0 & 3 & 1 & 2 
\end{pNiceArray}.$$
\end{example}

An $(A,b)$ basis that is also an ordinal basis of $C$ is called a \emph{dominating basis} for $(A,b,C)$. Given a dominating basis $B$ for $(A,b,C)$, the unique vertex $x$ of~\eqref{eq:Ax-leq-b} corresponding to $B$ is called a \emph{dominating vertex} for $(A,b,C)$. \emph{Scarf's lemma}~\cite{scarf1967core}, presented next, shows that such a vertex always exists.

\begin{theorem}[Scarf's Lemma]\label{thm:scarflemma}
Let $A\in \Q^{n \times (n+m)}_{\ge 0}$ be in standard form, $C \in \Z^{n \times (n+m)}$ be an ordinal matrix, and $b\in\Q^n_{+}$ such that~\eqref{eq:Ax-leq-b} is bounded. Then there exists a dominating vertex for $(A,b,C)$. \end{theorem}

Theorem~\ref{thm:scarflemma} was originally proved using a pivoting algorithm, described next.

\smallskip

\paragraph{Scarf's algorithm. }
Scarf's algorithm starts by letting $B=\{1,\cdots,n\}$ and $D=\{j_0,2,3,\cdots,n\}$, where $j_0$ is selected from the columns $k\in \{n+1,\dots,m\}$ so as to maximize $c_{1k}$. That is, 
$c_{1j_0}=\max_{k>n} c_{1k}$ (in Example~\ref{ex:utility-vector}, $j_0=\mathit{5}$ and $D=\{\mathit{5},\mathit{2},\mathit{3}\}$).
Note that $|B\cap D|\geq n-1$, $B$ is an $(A,b)$ basis, and $D$ is an ordinal basis of $C$. 

These properties are satisfied throughout the algorithm: at the beginning of each iteration, we have an $(A,b)$ basis $B$ and an ordinal basis $D$ with $|B \cap D|\geq  n-1$. If $|B\cap D|=n$, then $B=D$ and the algorithm halts and outputs the dominating basis $B$. Note that, in this case, $B$ is a dominating basis for $(A,b,C)$. Else, we let $\{j_t\}= D\setminus B$ and perform the following:
\begin{enumerate}
    \item[1]\emph{Cardinal pivot:} A column $j_\ell \in B \cap D$ is chosen, so that $B':=B\setminus \{j_\ell\} \cup \{j_t\}$ is a basis for $(A,b)$. That is, $j_t$ enters and $j_\ell$ leaves $B$.
    \item[2]\emph{Ordinal pivot:} A column $j^* \notin D$ is chosen, so that $D':=D \setminus \{j_\ell\} \cup \{j^*\}$ is an ordinal basis of $C$. That is, $j_\ell$ leaves and $j^*$ enters $D$.
\end{enumerate}
The algorithm then proceeds to the next \emph{iteration}, setting $B=B'$ and $D=D'$. Note that the change of bases in cardinal pivoting coincides with the classical pivoting operation employed by, e.g., the simplex algorithm. In particular, by basic linear algebra, there is always at least a feasible choice for $j_\ell$~\cite{bertsimas1997introduction}---even though, if the polytope is degenerate, multiple choices may be possible. \emph{Ordinal pivoting} is, at every step, uniquely determined~\cite{scarf1967core}. 

In non-degenerate polytopes, Scarf~\cite{scarf1967core} proved that the algorithm always terminates. For degenerate polytopes, a reduction to the non-degenerate case shows that a pivoting rule leading to convergence exists, but to the best of our knowledge, this pivoting rule is polytope-specific and does not match any of the ``standard'' pivoting rules investigated, e.g., in the literature on the simplex method~\cite{bland1977new}.
\smallskip

\paragraph{Applications to Stable Marriage.} As mentioned in the introduction, when specialized to certain polytopes, Scarf's lemma can be used to establish the existence of specific combinatorial objects. In this section, we review the setting that is relevant for our paper, that is, the stable marriage model. 

Let $M=\{m_1,\dots,m_k\}$ denote a set of men, and $W=\{w_1,\dots,w_k\}$ denote a set of women. Without loss of generality, we assume that $|M|=|W|=k$ and that every possible pair $(m_i,w_j)$ is acceptable: for any man $m\in M$, there is a strict linear order over $W\cup\{m\}$ such that $m$ is ordered last ($m$ being matched to $m$ means that $m$ is left unmatched). It is well-known that every stable marriage instance can be transformed to an instance with these properties without loss of generality~\cite{gusfield1989stable}. We refer to this order as the \emph{preference order of $m$} and denote it by $\succ_{m}$. The preference order for every woman $w$ is defined analogously and it is denoted by $\succ_{w}$. A \emph{matching} is a set of disjoint pairs $\mu\subset M\times W$. Given a matching $\mu$, a pair $(m,w)\in M\times W$ is a \emph{blocking pair} if both $w\succ_m \mu(m)$ and $m\succ_w \mu(w)$,  where for $v \in M\cup W$, we denote by $\mu(v)$ the partner of $v$ in $\mu$ if such a partner exists, or $\mu(v)=v$ otherwise. A matching $\mu$ is called \emph{stable} if no blocking pair exists. It has been observed, e.g., in Bir\'o and Fleiner~\cite{biro2016matching}, that Scarf's lemma can be used to show the existence of stable matchings in marriage instances.

\begin{theorem}\label{thm:existence-sm}
Consider an instance ${\cal I}$ of the marriage model defined over a bipartite graph $G(V,E)$, and let~\eqref{eq:Ax-leq-b} describe the (classical) matching polytope of $G(V,E)$, with $A'$ being the node-edge incidence matrix of $G$ and every component of $b$ being $1$. There exists a matrix $C$ such that every dominating vertex for $(A,b,C)$ is the characteristic vector of a stable matching.
\end{theorem}
\smallskip

For $A,b,{\cal I}$ as in Theorem~\ref{thm:existence-sm}, we say that $A,b$ are \emph{induced} by instance ${\cal I}$. 


\subsection{Our Contributions}\label{sec:contribution}

\paragraph{Polynomiality of Scarf's Algorithm in the Marriage Model.} As our first result, we show that Scarf's algorithm on an input $(A,b,C^*)$, where $A,b$ are as in Theorem~\ref{thm:existence-sm} and $C^*$ is a specific choice among the matrices $C$ that make Theorem~\ref{thm:existence-sm} true, can be implemented to run in polynomial time. In particular, as the matching polytope is degenerate, we find a pivoting rule (see Algorithm~\ref{alg:pivoting}) to control the cardinal pivots.
Bir\'o and Fleiner~\cite{biro2016fractional} ask whether Scarf's algorithm terminates in polynomial time for matching games, and our result gives a positive answer to their question in the special case of stable marriage games. 


\begin{theorem}\label{main:bipartite-poly}
For any instance ${\cal I}$ of the marriage model over a bipartite graph $G(V,E)$, there exists a cardinal pivoting rule such that Scarf's algorithm runs in polynomial time on input $(A,b,C^*)$, where $\{x \in \R^m_{\geq 0} : Ax =  b\}$ describes the matching polytope of $G$ and $C^*$ is a specific matrix that make Theorem~\ref{thm:existence-sm} true. In particular, the output will be the characteristic vector of a stable matching. 
\end{theorem}


As a building block to the proof of Theorem~\ref{main:bipartite-poly}, we develop an understanding of pivoting operations connecting ordinal and feasible bases that may be visited by Scarf's algorithm.

\paragraph{Polynomial-time Convergence Through a Perturbation of the Polytope.}

We show that polynomial-time convergence can also be achieved by perturbing the underlying bipartite matching polytope as to make it non-degenerate, see Section~\ref{sec:perturbation}. This result can be of computational interest, since on non-degenerate polytopes the behaviour of Scarf's algorithm is univocally determined, hence no ad-hoc pivoting rule needs to be implemented. Moreover, this latter theoretical result substantiates the empirical observation that Scarf's algorithm converges fast in standard perturbations of the bipartite matching polytope~\cite{biro2016fractional}. 



\paragraph{Limits of Scarf's Approach: Expressing Stable Matchings.}  Recent work has focused on understanding the ``expressive power'' of algorithms for constructing stable matchings.  For instance, while Gale and Shapley's algorithm~\cite{gale1962college} can be used to produce at most two stable matchings, the deferred acceptance algorithm with compensation chains can output all stable matchings~\cite{dworczak2021deferred}. Bir\'o and Fleiner~\cite{biro2016fractional} ask whether Scarf's algorithm can also be used to output all stable matchings of a given instance. 
As our next result, we show that the expressive power of Scarf's algorithm is also weak, since it will, in general, output only an exponentially small subset of stable matchings. 

Our result requires $C$ to be \emph{consistent}~\cite{aharoni2003lemma}. This is a common assumption that, roughly speaking, states that $C$ needs to characterize $\succ$ correctly (see Section~\ref{sec:overview-expressing-sm} for a formal definition). We emphasize that all the ordinal matrices $C$ discussed in 
Theorems~\ref{thm:existence-sm},~\ref{main:bipartite-poly}, and extensions to the roommate setting, to  hypergraphic matching~\cite{aharoni2003lemma}, and to matching with couples~\cite{biro2016matching}, satisfy consistency.  More generally, all applications of Scarf's algorithm to stable matching problems we are aware of employ a consistent matrix $C$. Now let $dom(A,b,C)$ be the family of dominating vertices of $(A,b,C)$ and, for a marriage instance ${\cal I}$, let ${\cal S}(I)$ be the family of stable matchings of ${\cal I}$. 

\begin{theorem}\label{thm:Scarf-is-weak}
There is a universal constant $c>1$ and, for infinitely many $n \in \N$, a marriage instance ${\cal I}_n$ with $n$ agents such that, for every $A,b$ induced by ${\cal I}_n$ and matrix $C$ consistent with ${\cal I}_n$, we have:
$$\frac{|{\cal S}({\cal I}_n)|}{|dom(A,b,C)|}=\Omega(c^n).$$
 \end{theorem} 

Hence, to cover all stable matchings of a marriage instance with dominating vertices of consistent matrices, we may need exponentially many matrices. In particular, for each consistent $C$, there are exponentially many stable matchings that cannot be obtained via Scarf's algorithm, thus answering a question of~\cite{biro2016fractional}.  

\subsection{Related Literature}\label{sec:related}
As mentioned in the introduction, Scarf's lemma has been used to show the existence of objects such as cores and fractional cores~\cite{biro2016fractional,scarf1967core}, strong fractional kernels~\cite{aharoni1995fractional}, and fractional stable solutions in hypergraphs~\cite{aharoni2003lemma} and stable paths~\cite{haxell2008fractional}. In particular, Bir\'{o} and Fleiner~\cite{biro2016fractional} investigate Scarf's algorithm for stable matching problems and extensions, posing many intriguing question on the features of Scarf's algorithm, some of which are investigated (and answered) in this paper, see Section~\ref{sec:contribution}. 
For some more complex markets, many $2$-stage rounding algorithms~\cite{nguyen2021stability,nguyen2018near,nguyen2019stable} use Scarf's algorithm as a first step to find a fractional point, which is then often rounded in a second step to a feasible or quasi-feasible solution.

To the best of our knowledge, no result on the polynomial-time convergence of Scarf's algorithm was known prior to this work. In contrast, it was known that the problem of finding a dominating vertex is PPAD-Complete, even in a restricted setting such as hypergraph matching~\cite{csaji2021complexity,ishizuka2018complexity,kintali2013reducibility}. PPAD~\cite{papadimitriou1994complexity} is a complexity class containing certain problems whose associated decision version always has a positive answer, but whose solution may be non-trivial to find. PPAD-Complete problems include the computation of Nash equilibria~\cite{chen2006settling,daskalakis2009complexity} and of fixed point of Brouwer functions~\cite{papadimitriou1994complexity}, among others, hence the existence of a polynomial-time algorithm that finds a solution for PPAD-Complete problems would be surprising. Moreover, examples are known where Scarf's algorithm's path is uniquely defined and requires an exponential number of steps (independently of any complex theoretic assumption), even if the corresponding dominating vertex can be found in polynomial time~\cite{EDMONDS20101281}. 

A crucial component of our approach is an understanding of the feasible and ordinal bases of the bipartite matching polytope that can be visited by Scarf's algorithm. In contrast, most of the polyhedral literature on matching polytopes has focused on studying vertices and conditions for their adjacency only (see, e.g.,~\cite{balinski1970maximum,sanita2018diameter}) in order to, e.g., bound the diameter, or to investigate classical pivoting operations~\cite{behrend2013fractional}. Another polyhedral approach to stable matching problems studies properties of the stable marriage or roommate polytopes, i.e., the convex hull of stable matchings, focusing on topics such as their linear descriptions~\cite{eirinakis2014polyhedral,faenza2022affinely,roth1993stable,rothblum1992characterization,teo1998geometry,vate1989linear} and diameters~\cite{eirinakis2014one}.

\addtocontents{toc}{\setcounter{tocdepth}{1}}

\section{Technical Overview}\label{sec:overview}
\subsection{Polynomiality of Scarf's Algorithm in the Marriage Model}\label{sec:hl-poly}

Consider a marriage instance $\mathcal{I}:=(G(V,E),\succ)$, where $G$ is a bipartite graph with nodes $V=M\cup W$, and $M$ and $W$ are the set of men and women, respectively. $E=E^\ell \cup E^v$, where $E^\ell$ is the set of \emph{loops} -- one for every vertex --  while $E^v$ is the set of \emph{valid} edges (i.e., not loops). Recall that we assume that the underlying graph is complete. $|V|=|E^\ell|=n=2k$ and $|E^v|=m=k^2$. For $v\in V$ and $e\in E$, we say $v$ is \emph{incident} to $e$ if $v\in e$. $\succ$ is a preference system such that for every $v\in V$, $\succ_v$ strictly ranks all the edges incident to $v$ and the unique loop on $v$ such that $e\succ_v (v,v)$ for every $e$ incident to $v$. We let $A=(I|A')$, where $A'$ is the incidence matrix of the graph, and the identity matrix $I$ corresponds to the loops (slack variables). We let $b=(1,1,\dots,1)^T\in\R^n$.

With $A, b$ as above,~\eqref{eq:Ax-leq-b} defines the matching polytope of $G(V,E)$. To a fixed graph $G(V,E)$ we associate a matrix $C$ such that each dominating vertex of $(A,b,C)$ is a stable matching of $G(V,E)$, see Theorem~\ref{thm:C-valid-for-marriage}. A set of columns of $A$ or $C$ (in particular, a feasible or ordinal basis) will also be interpreted as a set of edges of $G$.



We give here a high-level view of the behaviour of Scarf's algorithm on a generic iteration associated with a pair $(B,D)$ (cardinal basis and ordinal basis, respectively), and a combinatorial interpretation of intermediate vertices found on the way. 
We denote the valid (resp.~loop) edges in $B$ by $E_B^v$ (resp.~$E_B^\ell$), and define the graph (opti, with loops) $G_B=(V,E_B^v \cup E_B^\ell)$. Similarly, $G_D$ (resp.~$E_D$) is the subgraph (resp.~the subset) containing all and only  the edges in $D$. The pair associated after performing a cardinal and an ordinal pivoting starting from $(B,D)$ is denoted by $(B',D')$.


\smallskip

We next introduce definitions and properties of objects associated to a given iteration. The reader can follow those in Figure~\ref{fig:bipartite-evencycle} and Figure~\ref{fig:bipartite-example} for illustration.

\begin{enumerate}

\item {\bf Structure of basis $B$: Forest with single loops}. $F=(V,E_B^v)$ is a forest, and each connected component of $F$ contains exactly one edge from $E_B^\ell$ (i.e.,  a loop, see Lemma~\ref{lem:forest}). We say therefore that $G_B$ has a \emph{forest with single loops} structure. This fact is probably folklore, but we could not find a reference and for completeness we give a proof in Appendix~\ref{sec:app:proof-basis}.

\item {\bf Separator}. Define $m_{k+1}=m_1$. Each of the ordinal bases visited by the algorithm will have a unique separator. This is an agent $m_i \in M$ that, among other properties, satisfies the following: 
    \begin{enumerate}
    \item[(i)] the separator $m_i$ is incident to both the loop $e_i$ and some valid edge(s);
    \item[(ii)] for $i'=2,\dots, i-1$, $m_{i'}$ is incident to at least one valid edge of $G_D$ and no loop; 
    \item[(iii)] for $i'=i+1,\dots, k$, $m_{i'}$ is only incident in $G_D$ to the loop $e_{i'}$. 
    \end{enumerate}  
    See Proposition~\ref{prop:separator} and Section~\ref{sec:convergence}.


\item {\bf Utility vector.} Recall that the utility $u \in \Q^n$ is defined as $u_i=\min\{c_{ij} : j \in D\}$ $\forall i \in [n]$. If an ordinal pivot starting from a basis with associated utility $u$ leads to a basis associated to $u'$ with $u_i'>u_i$ for some $i \in [n]$, we say that the utility of $i$ \emph{increases}. 

\item {\bf $v$-disliked edge}, where $v \in M \cup W$. We say that $e_i \in E_D$ is \emph{$v_j$-disliked} w.r.t.~$D$ if $c_{ij}=u_i$.  This defines a bijection between edges in an ordinal basis $D$ and agents, see Definition~\ref{def:disliked}.
\end{enumerate}

\smallskip

In each iteration, we execute the following steps:

\begin{enumerate}
    \item {\bf Identification of man- and woman-disliked edges.} We identify the set of \emph{woman-} (resp.~\emph{man-}) \emph{disliked} edges as the set of edges that are $v$-disliked, for some agent $v$ that is a woman (resp.~man).
    
    \item {\bf Cardinal Pivoting:} Let ${\Lambda}$ be the connected component of $G_B$ containing the separator $m_i$.
    Recall that the next basis has the form $B'=B \cup \{j_t\}\setminus \{j_\ell\}$ and $G_{B'}$ is a forest with single loops. We show that one of the following holds:
    \begin{enumerate}
        \item\label{it:case-1-hl-bip} $\Lambda$ is a connected component of $G_{B \cup \{e_{j_t}\}}$ containing an even cycle $C$. Then $e_{j_\ell}$ can be chosen to be an edge of $C$. See Figure~\ref{fig:bipartite-evencycle}.
        \begin{figure}[h!]
    \centering
    \includegraphics{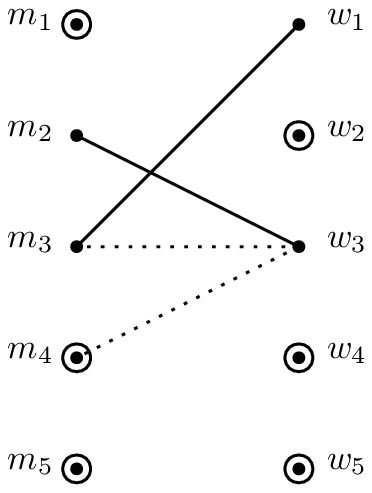}\hspace{1.5cm}
    \includegraphics{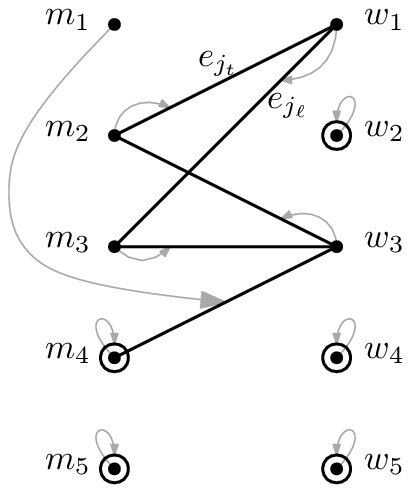}\hspace{1.5cm}
    \includegraphics{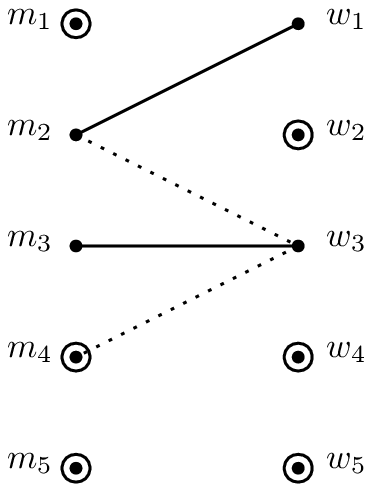}
    \caption{An illustration of some of the concepts introduced for the marriage case. We let $V=\{m_1,\dots,m_5;w_1,\dots,w_5\}$, and the pivoting rule when an even cycle occurs. On the left: The graph $G_B$, with the convention that solid edges are associated to variables from $B$ with $x$-value $1$ and dotted edges to variables from $B$ with $x$-value $0$. One can see that $G_B$ is a forest with single loops. Edge $e_{j_t}=(m_2,w_1)$ will enter the basis creating an $x$-alternating cycle $Q=(m_2,w_1),(w_1,m_3),(m_3,w_3),(w_3,m_2)$. In the center: The graph $G_D$,. All edges are full since there is no value associated to a ordinal basis. Gray arrows denote which edge is disliked by each node. It can be observed that $m_4$ is the separator. Edges $e_{j_t}$ (entering $B$) and $e_{j_{\ell}}$ (leaving $B$) are highlighted. In particular, the entering edge $e_{j_t}$ is $m_2$-disliked, and the addition of $e_{j_t}$ to $E_B$ creates a connected component with an even cycle.  On the right: The cardinal pivoting removes $e_{j_\ell}=(w_1,m_3)$, and leads to the new feasible basis $B'=B \setminus \{e_{j_\ell}\} \cup \{e_{j_t}\}$. Our pivoting rule indicates that, in this iteration, we can always select a woman-disliked edge inside $Q$ to leave. Notice that we may also select the leaving edge to be $(w_3,m_2)$. We arbitrarily select one when multiple choices exist.}
    \label{fig:bipartite-evencycle}
\end{figure}

        \item\label{it:case-2-hl-bip} In $G_{B \cup \{e_{j_t}\}}$, ${\Lambda}$ is joined with another connected component of $G_B$, as to create a component $\Gamma$ with two loops. Then $e_{j_\ell}$ can be chosen to be either a loop, or an edge of the path connecting the two loops. See Figure~\ref{fig:bipartite-example}.
        
        \begin{figure}[h!]
    \centering
    \includegraphics{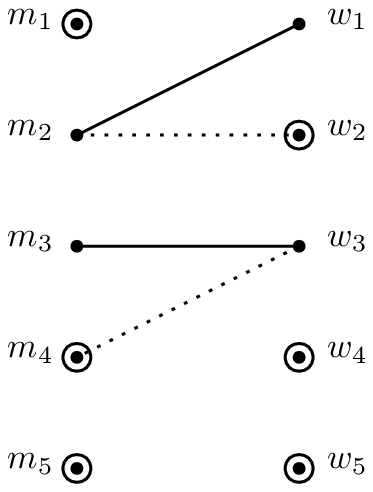}\hspace{1.5cm}
    \includegraphics{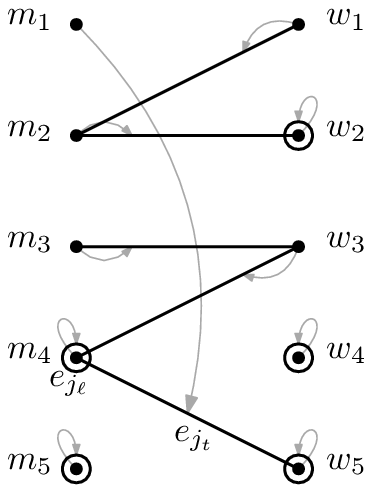}\hspace{1.5cm}
    \includegraphics{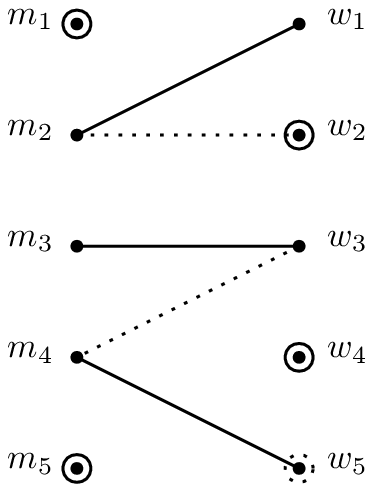}
    \caption{An illustration of pivoting rule when a path connecting two loops occurs. On the left: The graph $G_B$. In the center: The graph $G_D$. It can be observed that $m_4$ is the separator. The addition of $e_{j_t}$ to $B$ creates a connected component with two loops.     On the right: The cardinal pivoting removes $e_{j_\ell}$, which is the loop $(m_4,m_4)$, and leads to the new feasible basis $B'=B \setminus \{j_\ell\} \cup \{j_t\}$.}
    \label{fig:bipartite-example}

\end{figure}
\end{enumerate} 

We show in Lemma~\ref{lem:cardpivot} that if $e_{j_t}$ is man-disliked, we can always pick $e_{j_\ell}$ to be either (i) a woman-disliked edge in $D$ {(Figure~\ref{fig:bipartite-evencycle})}, or (ii) the loop corresponding to the separator $m_i$ {(Figure~\ref{fig:bipartite-example})}.

   \item {\bf Ordinal Pivoting:} 
    The new ordinal basis will be of the form $D'= D \setminus \{j_\ell\} \cup \{j^{*}\}$. Recall that $j^*$ is uniquely determined by $D, j_\ell$. In Section~\ref{sec:ordinal-pivot}, we show that if case (i) in the Cardinal Pivoting analysis holds, then $e_{j^*}$ is a $m_j$-disliked valid edge in $D'$ for some $j\in\{2,3,\dots,i-1\}$, see Figure~\ref{fig:bipartite-ordin-womandislikedenter}, while if (ii)  holds, then $e_{j^*}$ is an $m_{1}$-disliked valid edge in $D'$ and $m_{i+1}$ is the separator in $D'$ (let $m_{k+1}=m_1$), see Figure~\ref{fig:bipartite-ordin-separatorchange}. In both cases, $e_{j^*}$ is man-disliked, which provides the exact conditions for the cardinal pivot rule discussed above to apply.

\begin{figure}[h!]
    \centering
    \includegraphics{fig-bipartite/fig-bipartite-dislikedevencycle.eps}\hspace{1.5cm}
    \includegraphics{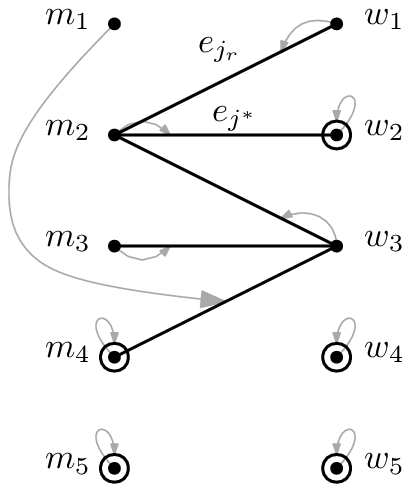}
    \caption{The ordinal pivot as a continuation of Figure~\ref{fig:bipartite-evencycle}. On the left: The graph $G_D$, where we want the $w_1$-disliked edge $(m_3,w_1)$ to leave. We then find the second worst choice for $w_1$, which is the reference edge $(m_2,w_1)$ and $m_2$-disliked in $D$. On the right: The graph $G_{D'}$, where the next entering variable $e_{j^*}$ is a $m_2$-disliked valid edge. During this ordinal pivot, the utility of $w_1=i_\ell$ increases, and the utility of $m_2=i_r$ decreases, while others do not change their utilities and their disliked edges. The separator stays at $m_4$.}
    \label{fig:bipartite-ordin-womandislikedenter}
    
\end{figure}
\medskip  
  
\begin{figure}[h!]
    \centering
    \includegraphics{fig-bipartite/fig-bipartite-disliked.eps}\hspace{1.5cm}
    \includegraphics{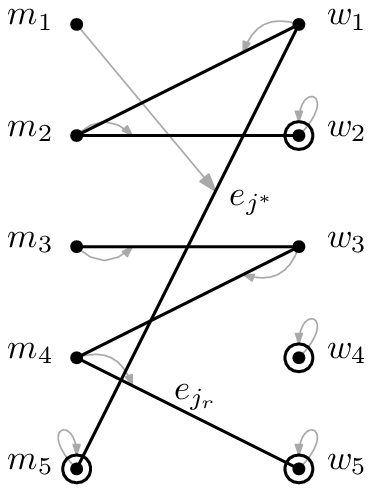}
    \caption{The ordinal pivot as a continuation of Figure~\ref{fig:bipartite-example}. On the left: The graph $G_D$, where we want the loop $(m_4,m_4)$ to leave. We then find the $m_1$-disliked edge $(m_4,w_5)$ as the reference edge $e_{j_r}$. On the right: The graph $G_{D'}$, where the next entering variable $e_{j^*}$ is a $m_1$-disliked valid edge. During this ordinal pivot, the utility of $m_4=i_\ell$ increases, and the utility of $m_1=i_r$ decreases, while others do not change their utilities and their disliked edges. The separator changes from $m_4$ to $m_5$.}
    \label{fig:bipartite-ordin-separatorchange}
    
\end{figure}
\medskip

\end{enumerate}
We then continue the next iteration with $(B',D')$. 

\smallskip

Our convergence analysis follows by the \emph{componentwise monotone evolution} of a \emph{potential vector} in $\Z^2$, see Section~\ref{sec:convergence}: \begin{equation}\label{eq:app-convergence-bi}
    \left(i,\sum_{w\in W}u_w \right),
\end{equation}
where $i$ is the index of the current separator. That is, at each iteration one component of the potential vector strictly increases while the other does not decrease. In particular, either the new separator becomes $m_{i+1}$ and and the total utility of women (i.e.,  the $\ell_1$ norm of the subvector of $u$ restricted to women) does not decrease, or the total utility of women strictly increases, while the separator is still $m_i$. 
Since both the number of separators (men) and the total utility of women are bounded by $O(n^2)$, the number of iterations is $O(n^2)$, and each iteration can clearly be performed in polynomial time. 

\smallskip

For a \emph{combinatorial interpretation} of the pivoting rule, one can think of the algorithm as adding men one by one from the queue $m_2,m_3,\dots,m_k,m_1$. At each iteration, the separator denotes the last man introduced. One can show that we change the separator from $m_i$ to $m_{i+1}$ as soon as we obtain a ``local'' stable matching, i.e., a stable matching restricted to men $m_2,m_3,\dots,m_i$ and all women. In iterations when the separator does not change, the algorithm adjusts the current matching by ``improving'' the matching for women (thus the increase in $\sum_{w \in W} u_w)$, until a ``local'' stable matching is obtained.

\subsection{Polynomial-time Convergence Through a Perturbation of the Polytope}\label{sec:overview:perturbation}

In the previous section, we showed that a stable matching can be obtained in polynomial time by running Scarf's algorithm on the bipartite matching polytope $\mathcal{P}$ with a suitable pivoting rule. The implementation of this approach would however be non-trivial, since it needs in particular  to deal with degenerate pivots. For practical purposes, it would be desirable to run Scarf's algorithm on a perturbation of ${\cal P}$, since every pivot of Scarf's algorithm on non-degenerate polytopes is uniquely determined~\cite{scarf1967core}, hence no tailored pivoting rule needs to be implemented.

In Section~\ref{sec:perturbation}, we show how a ``classical'' perturbation allows us to find a stable matching using Scarf's algorithm in polynomial time. Our approach is as follows. We perturb $\mathcal P$ as to construct a non-degenerate polytope $\mathcal{P}^\epsilon$ (see~\eqref{eq:nondegeneratepoly}). Thus, an execution of Scarf's algorithm on $\mathcal{P}^\epsilon$ is uniquely defined by a sequence 
of pairs $(x^\epsilon_0,D_0)\to(x^\epsilon_1,D_1)\to\dots\to (x^\epsilon_N,D_N)$, where $x_I^\epsilon$ is a vertex of $\mathcal{P}^\epsilon$ and $D_I$ is an ordinal basis of $C$, for $0\le I\le N$. Our main technical ingredient here is to show that there is a sequence of Scarf pairs $(B_0,D_0)\to(B_1,D_1)\to\dots\to (B_N,D_N)$ satisfying our pivoting rule in the non-perturbed case (i.e., Algorithm~\ref{alg:pivoting}) and such that, for each $I$, $B_I$ corresponds in ${\cal P}^\epsilon$ to $x_I^\epsilon$. 
Therefore, based on Theorem~\ref{main:bipartite-poly}, we have $N=poly(n)$, achieving the claimed polynomial-time convergence. The details are given in Theorem~\ref{thm:perturbation}.

\subsection{Limits of Scarf's Algorithm: Expressing Stable Matchings}\label{sec:overview-expressing-sm}

As we argue next, under the (quite natural) consistency assumption on $C$, the set of stable matchings that Scarf's algorithm can output may be exponentially smaller than the set of all stable matchings of an instance. 
\begin{definition}\label{def:C-consistency}
An ordinal matrix $C$ is \emph{consistent} with a marriage instance ${\cal I}=(G(V,E),\succ)$ with $V=\{v_1,\dots,v_n\}$ if

\begin{enumerate}
    \item[(i)] For any $i, \bar \imath \in [n]$ and $e_j\in E$ such that $v_i\in e_j$ and $v_{\bar \imath}\notin e_j$, we have $c_{i,j}<c_{\bar \imath,j}$.
    \item[(ii)] For any $i\in [n]$ and $e_j,e_\ell\in E$ such that $v_i\in e_j,e_\ell$, then $c_{ij}>c_{i\ell}$ if and only if $e_j\succ_{v_i}e_\ell$.
\end{enumerate}
\end{definition}
In Definition~\ref{def:C-consistency}, (i) is a regularity condition, and (ii) means that $C$ characterizes the order of $\succ$ correctly. Under the assumption that $C$ is consistent, one can deduce that every dominating vertex $x$ of $(A,b,C)$ is the characteristic vector of  a stable matching of ${\cal I}$, where $A,b$ is induced by ${\cal I}$, see~\cite{aharoni2003lemma}. We emphasize that all the ordinal matrices $C$ discussed in 
Theorems~\ref{thm:existence-sm},\ref{main:bipartite-poly}, and extensions to hypergraphic matching~\cite{aharoni2003lemma} and matching with couples~\cite{biro2016matching} yield consistent $C$. 

We show that, when $C$ is consistent, $dom(A,b,C)$ always yields a $v$-optimal stable matching, i.e., there exists an agent $v$ who is matched to her best partner among all stable matchings. We remark, in passing, that, since our implementations of Scarf's algorithm discussed in Section~\ref{sec:hl-poly} and Section~\ref{sec:overview:perturbation} are obtained with a consistent $C$, the theorem applies for those algorithms as well.
\begin{theorem}\label{main:lastoptimal}
Suppose $C$ is a consistent ordinal matrix and $A,b$ are as in Theorem~\ref{thm:existence-sm}. Then any dominating vertex of $(A,b,C)$ is a characteristic vector of a $v$-optimal stable matching for some $v\in V$.\end{theorem}

Hence, any ``intermediate'' stable matching (i.e.,  that does not assign any agent to their favorite stable partner) cannot be output by Scarf's algorithm because it cannot even be represented by a dominating vertex. We show in Example~\ref{ex:Irving-leather} an infinite family of instances $\mathcal{I}_n$ with $2$ stable matchings that are $v$-optimal for some $v$, but exponentially many stable matchings, and Theorem~\ref{thm:Scarf-is-weak} follows.


\section{Review of Scarf's Algorithm}\label{sec:ReviewScarf}

We discuss here some classical properties of Scarf's algorithm and introduce related definitions and notation. This section can be used by the reader as an introduction / reminder of the algorithm, but it also presents building blocks that will be used in our arguments in future sections. Missing proofs can be found in~\cite{scarf1967core}. 

Recall that the input to Scarf's algorithm is given by $n\times (n+m)$ nonnegative matrices $A$ and $C$ with special properties. We call $C$ an \emph{ordinal} matrix (see Section~\ref{sec:scarf-lemma}). 

\subsection{Cardinal and Ordinal Pivots} To make the argument clear, we first add the standard nondegeneracy assumption that all of the variables associated with the $n$ columns of a feasible basis $B$ for the equations $Ax=b$ are strictly positive (Scarf's algorithm was originally stated in this setting only~\cite{scarf1967core}). Recall that at each step of Scarf's algorithm, we are given matrices $B$ and $D$, where $B$ is a $(A,b)$ basis and $D$ is an ordinal basis for $C$. The properties of $D$ imply the following.

\begin{proposition}\label{prop:app1.1}
For any column $c$ of an ordinal basis $D$, there is exactly one row minimizer to be used in forming the utility vector. More formally, there is a unique row $i \in [n]$, such that $u_{i}=c_{i}$. Hence, for the other rows $\bar \imath\neq i$, we have $u_{\bar \imath}<c_{\bar \imath}$.\end{proposition}

Proposition~\ref{prop:app1.1} gives a bijection from the $n$ columns of an ordinal basis to the $n$ rows.

Recall that an iteration of Scarf's algorithm consists of two main steps: cardinal pivot and ordinal pivot. 
Cardinal pivot is similar to the pivot performed by the simplex algorithm, where we have:
\begin{lemma}\label{lem:appcplemma}
Let $B=\{j_1,\cdots,j_n\}$ be an $(A,b)$ basis, and let $j'$ be an arbitrary column not in $B$. Then, if~\eqref{eq:Ax-leq-b} is nondegenerate and the feasible set $\{x|x\ge 0 \textrm{ and } Ax=b\}$ is bounded, there is a unique $j_t \in B$ such that $B\setminus \{j_t\} \cup \{j'\}$ is an $(A,b)$ basis.
\end{lemma}
The previous is a standard result in linear programming, which says we can arbitrarily choose an outside column to enter the basis while a unique column leaves. A symmetric property holds for the ordinal pivot. An arbitrary column in an ordinal basis visited by the algorithm may be removed and a unique column introduced from outside so that the new set of columns is also an ordinal basis.
\begin{lemma}\label{lem:appoplemma}
Let $D=\{j_1,\cdots,j_n \}$ be an ordinal basis of $C$ and $\ell \in [n]$. Then there exists a unique column $j^*\notin D$ such that $D\cup \{j^*\}\setminus \{j_\ell\}$ is an ordinal basis of $C$. 
\end{lemma}
Recall that the procedure to find such a new column $j^*$ is called ordinal pivot. We give the formal definition of ordinal pivot as follows:
\begin{definition}[Ordinal pivot]\label{def:appopdef}
Consider an ordinal basis $D$ and a specific column $j_\ell$ to be removed from it. In the $n\times (n-1)$ matrix of remaining columns, define $$\bar{u}_i=\min_{j\in D\setminus \{j_\ell\}} c_{i,j}.$$
There exists exactly one column $j_r$ that contains two row minimizers in forming $\bar{u}$ (according to Proposition \ref{prop:app1.1}). We call $j_r$ the \emph{reference column}. Between the two minimizers, one of them is new and the other is a row minimizer of the original ordinal basis. Let the row associated with the former have an index $i_\ell$ (i.e., $c_{i_\ell,j_\ell}=u_{i_\ell}<\bar{u}_{i_\ell}=c_{i_\ell,j_r}$) and the latter have an index $i_r$ (i.e., $c_{i_r,j_r}=u_{i_r}=\bar{u}_{i_r}$). Let $K$ denote the columns in $C$ such that $k\in K$ if 
\begin{equation}\label{eq:appchoiceoford}
    c_{i,k}>\bar{u}_i, \textrm{ for all } i\neq i_r.
\end{equation}
Of the columns in $K$, select the one which maximizes $c_{i_r,k}$, i.e.
$$j^*=\mathop{\arg\max}_{k\in K} c_{i_r, k}.$$
An ordinal pivot step introduces this column $j^*$ into the ordinal basis, as to form the new ordinal basis $D'=D\setminus \{j_\ell\} \cup \{j^*\}$. 
\end{definition}

It can be shown that $j=j^*,j_\ell$ are the only two columns that make $D\setminus\{j_\ell\}\cup\{j\}$ an ordinal basis. This fact tells us that the number of ordinal bases which contains any given $n-1$ columns can only be $0$ or $2$. It also suggests that the ordinal pivots are ``reversible'': If $j_\ell$ is eliminated from a basis and $j^*$ brought in, then $j^*$ may be eliminated from the new basis and the original basis will be obtained.\par
It is useful to analyze the change of utility vector in an ordinal pivot. We follow the notation from Definition \ref{def:appopdef}.
\begin{lemma}\label{lem:app:changeofu}
Consider the two utility vectors $u$ and $u'$ associated to $D$, $D'$, respectively. Then, when going from $D$ to $D'$, the utility of $i_\ell$ increases while the utility of $i_r$ decreases, and others are indifferent. Formally, 
$$u'_{i_\ell}=c_{i_\ell,j_r}>c_{i_\ell,j_\ell}=u_{i_\ell},$$
$$u'_{i_r}=c_{i_r,j^*}<c_{i_r,j_r}=u_{i_r},$$
$$u'_{i}=u_{i},\textrm{ for $i \in [n],\  i \neq i_\ell,i_r$}.$$
\end{lemma}

\subsection{Scarf Pairs, Almost-feasible Ordinal Bases, and the Termination Condition for the Algorithm}\label{sec:scarf:almost-feasible-etc}

We have already discussed in Section~\ref{sec:scarf-lemma} how the sequence of ordinal and cardinal pivoting leads, in finite time, to a dominating vertex. We next give more details on the (ordinal) bases visited and the termination condition of the algorithm.

\begin{definition}[Scarf pair]\label{def:iter-scarfpair}
Suppose Scarf's algorithm starts with initial feasible basis $B_0$ and ordinal basis $D_0$. For $i\ge 1$, if $B_{i-1}\neq D_{i-1}$, then the $i$-th \emph{iteration} consists first of a cardinal pivot $(B_{i-1},D_{i-1})\to (B_i,D_{i-1})$ and then, if $B_i\neq D_{i-1}$, of an ordinal pivot $(B_{i},D_{i-1})\to (B_i,D_i)$, where the pivots are defined in the previous section. Let $$B_0,B_1, \dots, B_I \quad \hbox{ and } \quad D_0,D_1,\dots, D_I$$ be the sequence of feasible (resp.~ordinal) bases visited by Scarf's algorithm such that $B_I=D_I$ when it terminates. For $0\leq i\leq I-1$, we call $(B_i,D_i)$ a \emph{Scarf pair}. 
\end{definition}

Hence, an iteration always starts with a Scarf pair and ends up with a new Scarf pair if the algorithm does not terminate. By Lemma~\ref{lem:appcplemma} and Lemma~\ref{lem:appoplemma}, any Scarf pair $(B_i,D_i)$ has $|B_i\cap D_i|=n-1$.

\begin{lemma}\label{lem:scarf-halts}
Throughout the algorithm, $B_{i-1}\neq D_{i-1}$ and $B_i=D_i$ if and only if in the $i$-th iteration column $1$ leaves the basis $B_{i-1}$ or column $1$ is introduced in the ordinal basis $D_{i-1}$. One of these two occurrences happen after a finite number of iterations, i.e., $I<\infty$ in Definition~\ref{def:iter-scarfpair}. In particular, a dominating vertex exists.
\end{lemma}

We next present properties shared by ordinal bases visited by our algorithm. 

\begin{definition}[Almost-feasible ordinal bases]\label{almostfeasibledef}
An ordinal basis $D=\{j_1,\dots,j_n\}$ is \emph{almost-feasible} if $1\notin D$, and there exists a feasible basis $B$ such that $B=\{1,j_1,\dots,j_{t-1},j_{t+1},\dots,j_n\}$ for some $t \in [n]$. That is, $1\in B$ and $R=B\cap D$ has cardinality $n-1$. We say that $B$ is \emph{associated} to $D$ and call the set $R$ \emph{remaining columns}, and let $A_R,D_R$ denote the submatrices of $A,C$ corresponding to the index set $R$, respectively. 
\end{definition}

By Lemma~\ref{lem:scarf-halts}, Scarf's algorithm terminates if column $1$ enters the ordinal basis $D$ or leaves the feasible basis $B$. Moreover, throughout the algorithm, for each Scarf pair $(B,D)$, by construction $B$ and $D$ differ in at most one entry. Hence, we deduce the following. 

\begin{lemma}
Let $D$ be an ordinal basis visited by some execution of Scarf's algorithm on the input from Theorem~\ref{thm:scarflemma}. Then either $D$ is the final basis visited by the algorithm, or $D$ is almost-feasible. 
\end{lemma}

\subsection{Dealing with Degeneracy}\label{sec:scarf:degeneracy}

Notice that in the framework above, the behaviour of Scarf's algorithm is uniquely determined by the input. In contrast, when Scarf's algorithm is applied to a degenerate polytope, i.e.,  such that there are strictly more than $n$ constraints that are tight at some vertex, then many options may be possible for cardinal pivoting. Hence, the output of Scarf's algorithm is not unique and depends on the cardinal pivoting rule. Moreover, cycling may happen. However, definitions and properties from Section~\ref{sec:scarf:almost-feasible-etc} apply, with the exception of finite convergence, which is not guaranteed. 

\begin{lemma}\label{lem:scarf-halts-2}
$B_{i-1}\neq D_{i-1}$ and $B_i=D_i$ if and only if in the $i$-th iteration column $1$ leaves the basis $B_{i-1}$ or column $1$ is introduced in the ordinal basis $D_{i-1}$.
\end{lemma}

In Section~\ref{sec:algo-bipartite}, we deal with matching polytope that are highly degenerate, and we show that with our cardinal pivoting rules Scarf's algorithm converges, and does so in polynomial time. 

There is however another standard way to deal with degeneracy: perturbation (this is discussed for the marriage case in Section~\ref{sec:perturbation}). Indeed, perturb the right hand side vector to $b'$ at the beginning to make the polytope nondegenerate. Then, Scarf's algorithm outputs a dominating basis $B$ w.r.t.~$(A,b',C)$. This implies that $B$ is an ordinal basis for $C$. If $B$ is also an $(A,b)$ basis, then $B$ is an $(A,b)$ basis~\cite{bertsimas1997introduction}, then it is dominating for $(A,b,C)$.
 Therefore, if the perturbation is small enough so that every $(A,b)$ is also an $(A,b')$ basis, then the output basis corresponds to a dominating vertex of the original polytope. 

\section{Additional Facts and Notations}\label{sec:notation}

We now discuss some facts and notation that are used throughout the rest of the paper. 

\subsection{Graphs, Edges, Matching}\label{sec:def:graphs-etc}

Recall that we start with the input $\mathcal{I}=(G(V,E),\succ)$, where $G$ is a graph with a set $V$ of $n$ nodes, $E$ is the set of edges given by the (disjoint) union of  $n$ \emph{loops} $E^\ell$ and $m$ valid edges $E^v$. For $v\in V$ and $e\in E$, we say $v$ is incident to $e$ if $v\in e$. $\succ$ is a preference system such that for every $v\in V$, $\succ_v$ ranks all the edges incident to $v$ and the unique loop on $v$ such that $e\succ_v (v,v)$ for every $e\in E^{v}$ incident to $v$. When $e=(v,u), e'=(v,u')$, we also write $u \succ_v u'$ when $e \succ_v e'$. 

For a graph $G(V,E)$, define the \emph{degree} of a node $v \in V$ as the number of edges in $G$ (including loops) incident to $v$, and denote it by $\deg_G(v)$. Note that, for the same node $v\in V$, the degree is subject to $G$. For example, $\deg_{G_B}(v)$, $\deg_{G_D}(v)$ may be different.

We can then redefine the concept of matching. We say that an edge $e$ is \emph{incident} to a node $v$ if the latter is an endpoint of the former. 
A  \emph{matching} is an edge set $\mu$ such that for each node $v$, there is \emph{exactly one} edge $e\in\mu$ incident to $v$. Equivalently, $\mu\subset E$ is a matching if the subgraph $G_\mu=(V,\mu)$ is such that every node $v \in V$ has $\deg_{G_\mu}(v)=1$. We say $\mu$ \emph{properly matches} $v$ if a valid edge in $\mu$ is incident to $v$. Hence, if a matching $\mu$ does not properly match node $v$, then the loop $e\in \mu$, which implies $v$ is unmatched in the classical sense.

\subsection{Bases and Related Objects}

Because of the structure of our input, the matrix $A$ (resp.~$C$) has exactly one row per agent, and exactly one column per edge, including loops. Therefore, we sometime overload notation and use the same symbol to denote an agent/edge and a row/column, with the exact meaning being always clear from the context. More precisely:

\begin{itemize}
    \item $m$, $w$, $v$ (possibly with subscripts/superscripts) are used to refer to agents (with $m \in M$, $w \in W$ and $v \in V$) or to the corresponding node or row index in either the $A$ or $C$ matrix. 
    \item $(m,m)$, $(w,w)$, $(m,w)$, $(v_i,v_\ell)$ denote edges or the corresponding column indices of $A$ or $C$. Note that we adopt the ordered pair notation even though all graphs are undirected. This is motivated by that fact that the ordered pair allows us to distinguish between men (first entry) and woman (second entry). We sometime also denote an edge by $e_j$, with $j\le n$ for loops and $j>n$ for valid edges.  Accordingly, $a_j$, $c_j\in \Q^n$ are the column vectors corresponding to $e_j$ in matrix $A$, $C$, respectively.
\end{itemize}

We similarly overload notation for bases / ordinal bases as follows:

\begin{itemize}    
    \item $B$ is some feasible basis for $A$. It can be an index set of $n$ columns, or of $n$ edges of $G$ (including loops). $A_B$ is the submatrix obtained from $A$ by restricting to columns of $B$. $G_B=(V,E_B)$ is a subgraph of $G$ only maintaining the edges in $B$.

    \item $D$ is some ordinal basis of $C$. It can be an index set of $n$ columns, or of $n$ edges of $G$ (including loops). $C_D$ is the submatrix obtained from $C$ by restricting to columns of $D$. $G_D=(V,E_D)$ is a subgraph of $G$ only maintaining the edges in $D$.
\end{itemize}        

Other relevant notation includes vectors associated to bases:    
\begin{itemize}
\item $x\in\R^{n+m}$ is a feasible solution of the desired polytope. It assigns weights on every edge on $E$. The \emph{$x$-value} of an edge, $x_e$, is the value of the component of $x$ corresponding to the edge $e$. 
\item For a feasible basis $B$, let $x$ be the basic feasible solution associated with $B$. Define $\mu_B$ as the matching with edges given by $supp(x)$, i.e., $e\in \mu_B$ iff $x_e>0$.
    \item For an ordinal basis $D$, let $u_D\in \Z^n$ denote the utility vector associated with $D$, i.e., $(u_D)_i=\min_{j\in D} c_{ij}$. If there is no ambiguity, we omit the subscript $D$ and just denote by $u$ the utility vector of $D$.
\end{itemize}

\subsection{Paths, Cycles}\label{sec:def:paths-etc}

We redefine the concept of path in graphs as follow. Let $F\subset [n+m]$ be a subset of columns and $A_F$ be the submatrix of $A$ restricted to columns in $F$. 
    Consider the graph $G_F=(V,E_F)$ where $E_F$ is the edge set corresponding to columns in $F$.
    A path $P$ in $G_F$ is defined as a sequence of edges
    $$P=(v_1,v_2),(v_2,v_3),\dots,(v_{\ell-1},v_\ell)$$
    with $v_1\neq v_\ell $, such that the two nodes in some parentheses can coincide, no pair is repeated, and the valid edges in $P$ form a path in the classical sense. Hence, paths may contain loops. A cycle $Q$ is defined as a cycle in the classical sense. We denote by $V_P$ and $E_P$ the vertices and edges of a path, respectively. A similar notation is employed for $Q$. Sometimes, in order to stress its vertices, when denoting a path $P$ (resp.~a cycle $Q$) we also add between pairs of consecutive edges the vertex they share.

    Let $x\in \R^{n+m}$. We say $P$ (resp., $Q$) is \emph{$x$-alternating} if the $x$-value of edges on $P$ (resp., $Q$) alternate between $0$ and $1$. We say $P$ is \emph{$x$-augmenting} if it is $x$-alternating and the first edge is a loop with $x$-value $1$, i.e., $v_1=v_2$, and $x_{(v_1,v_1)}=1$. An illustration of the concepts of path and cycle is given in Figure~\ref{fig:bipartite-pathandcycle}.

\subsection{$v$-disliked Edges}

Thanks to the additional properties of Scarf's algorithm discussed in Section~\ref{sec:ReviewScarf}, we can now formalize the concept of \emph{disliked edge} first mentioned in Section~\ref{sec:hl-poly}. 

\begin{definition}[$v$-disliked Edge] \label{def:disliked}
For an ordinal basis $D$ and $i\in[n]$, let $j\in D$ be such that $u_i=c_{ij}$, where the existence and uniqueness of $j$ follows from Proposition~\ref{prop:app1.1}. We call $e_j$ the \emph{$v_i$-disliked} edge (in $D$). 
\end{definition}

In other words, there is a bijection between the rows (nodes) and columns (edges) given by $D$, and we denote this bijection by saying that $v_i$ \emph{dislikes} $e_j$. The term ``dislikes'' comes from the fact that, if one interprets entries in $C$ as $v_i$'s evaluation of \emph{all} edges of $G$ (including those not incident to $v_i$), then the $v_i$-disliked edge will be $e_{\bar \jmath}$ for the unique minimizer $\bar \jmath$ of $c_{ij}$ over all $j \in D$ -- hence the edge achieving the worst evaluation according to $v_i$. An important observation is that the bijection is subject to $D$: When we say $v_i$ dislikes $e_j$, we need to specify which $D$ is referred to -- but if there is no ambiguity, we omit $D$.

\begin{figure}
    \centering
    \includegraphics[scale=1.4]{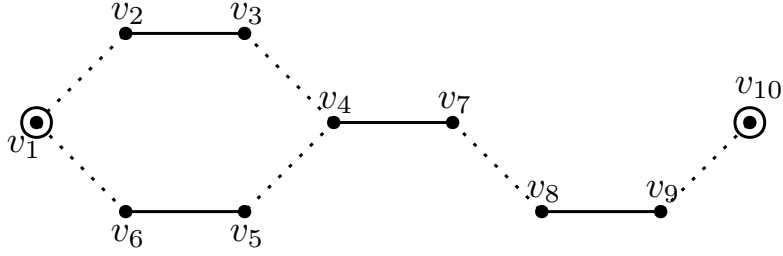}
    \caption{An illustration of the concepts of path and cycle, with the convention that solid edges (resp.~dotted edges) are associated to variables from $F$ with $x$-value $1$ (resp.~$x$-value $0$). $P=(v_1,v_1),(v_1,v_2),(v_2,v_3),(v_3,v_4),(v_4,v_7),(v_7,v_8),(v_8,v_9),(v_9,v_{10}),(v_{10},v_{10})$ is an $x$-augmenting path, and any sequential subset of $P$ is an $x$-alternating path. $Q=(v_1,v_2),(v_2,v_3),(v_3,v_4),(v_4,v_5),(v_5,v_6),(v_6,v_1)$ is a cycle. $Q$ is not $x$-alternating since there are consecutive $0$'s such as $(v_3,v_4),(v_4,v_5)$.}
    \label{fig:bipartite-pathandcycle}
\end{figure}

\section{Polynomiality of Scarf's Algorithm in the Marriage Model}\label{sec:algo-bipartite}

\subsection{The Bipartite Matching Polytope and an Ordinal Matrix Associated to a Marriage Instance}\label{sec:matrixdesign}

We introduce here the ordinal matrix $C$ used in Theorem~\ref{main:bipartite-poly} and its features. Interestingly, $C$ is a special case of the matrix defined by~\cite{biro2016matching}.

Consider a complete instance ${\cal I}$ with $k$ men and $k$ women, denoted respectively be $M=\{m_1,\dots,m_k\}$ and $W=\{w_1,\dots,w_k\}$ and arranged in a complete graph $G$. For $i \in [k]$, we also denote $m_i$ by $v_i$ and $w_i$ by $v_{k+i}$. Recall $n=2k$, $m=2k+k^2$. For edges, we let $e_1,\dots, e_n$ denote the loops containing nodes $v_1,\dots,v_n$ and for $i \in [n]$, $e_{k(i+1)+1},e_{k(i+1)+2},\dots,e_{k(i+1)+k}$ denote the edges that connect $m_i$ and women following the decreasing order of $m_i$'s preference. 

Hence, with the definition of edge given in Section~\ref{sec:def:graphs-etc},~\eqref{eq:Ax-leq-b} defines the matching polytope of $G$ when the entry $(i,j)$ of the $A$ matrix is given as follows: \begin{displaymath}
a_{ij}=\left\{
\begin{array}{cc}
   1,  & \textrm{if $e_j$ is incident to $v_i$} \\
   0,  & \textrm{others}
\end{array}.
\right.
\end{displaymath}

For the construction of $C$, we define
$$\mathcal{S}=\{0\}, \mathcal{M}=\{1,\dots,k\}, \mathcal{L}=\{k+1,\dots,k^2\}, \mathcal{ XL}=\{k^2+1,\dots,k^2+2k-1\}.$$
For any row $i$, assign 
$$c_{ij}=k+1-\ell\textrm{, if $a_{ij}=1$ and $e_j$ is the $\ell$-th best candidate for node $v_i$}.$$
This implies $c_{ii}=0\in \mathcal{S}$, and $c_{ij}\in \mathcal{M}$ if $e_j$ is a valid edge incident to $v_i$. All entries $c_{ij}$ where $j$ does not correspond to edges incident to $i$ have values in ${\cal L}\cup {\cal XL}$, as discussed below.

If $e_j$ is a valid edge but not incident to $v_i$, then assign $c_{ij}$ a number in $\mathcal{L}$. We have $k(k-1)$ such edges, matching the cardinality of $\mathcal{L}$. We distribute the numbers in $\mathcal{L}$ according to decreasing order from left to right. In other words, if $e_j,e_\ell$ are not incident to $v_i$ and $j,\ell>2k$, then $c_{ij},c_{i\ell}\in \mathcal{L}$ and $c_{ij}>c_{i\ell}$ if and only if $j<\ell$.\par
If $e_j$ is a loop, but $j\neq i$, we assign $c_{ij}$ a number in $\mathcal{XL}$. Similarly, we distribute the numbers in $\mathcal{XL}$ according to a decreasing order from left to right on those vacant positions. Note that $C$ is a \emph{consistent} ordinal matrix, according to the definitions given in Section~\ref{sec:scarf-lemma} and Section~\ref{sec:overview-expressing-sm}.

\begin{example}\label{ex:bi-Cmatrix} Given the instance with $k=2$ and the following preference lists:
$$m_1: w_2\succ w_1, \quad m_2: w_1\succ w_2, \quad  w_1: m_2\succ m_1, \quad 
w_2: m_1\succ m_2,$$
we can construct the $C$ matrix:
$$
\begin{pNiceMatrix}[first-row,first-col][columns-width=1.2cm]
    & m_1 & m_2 & w_1 & w_2 & (m_1,w_2) & (m_1,w_1) & (m_2,w_1) & (m_2,w_2)\\
    m_1 & 0 & 7 & 6 & 5 & 2 & 1 & 4 & 3\\
    m_2 & 7 & 0 & 6 & 5 & 4 & 3 & 2 & 1\\
    w_1 & 7 & 6 & 0 & 5 & 4 & 1 & 2 & 3\\
    w_2 & 7 & 6 & 5 & 0 & 2 & 4 & 3 & 1
\end{pNiceMatrix}.
$$
In this instance, $\mathcal{L}=\{3,4\}$, $\mathcal{XL}=\{5,6,7\}$, and they are assigned in decreasing order from left to righht on each row.\end{example}

The way we distribute the numbers in $\mathcal{L}$, $\mathcal{XL}$ is irrelevant to the next theorem, so there are multiple $C$'s for which the following result is valid. However, we set them as discussed above in order to: (i) satisfy the ordinal matrix conditions (cf. Section~\ref{sec:scarf-lemma}); (ii) provide a clear structure of ordinal bases, on which we will build upon in the next section. Since matrix $C$ is a special case of the matrix from Theorem~\ref{thm:existence-sm} from~\cite{biro2016matching}, we conclude the following.

\begin{theorem}[\cite{biro2016matching}]\label{thm:C-valid-for-marriage}
Any dominating basis for (A,b,C) defined above corresponds to a stable matching of ${\cal I}$.
\end{theorem}

The design and analysis of the algorithm follow the high-level view of the polynomial-time pivoting rule given in Section~\ref{sec:hl-poly} and are organized as follows. In Section~\ref{sec:basic-of-scarf}, we introduce definitions and show the forest with single loops structure of basis of the matching polytope. In Section~\ref{obsec} we discuss properties that ordinal basis will have throughout the algorithm, which allow us to define the separator and investigate the disliked relation. Features of cardinal and ordinal pivots are presented in Section~\ref{sec:cardinal-pivot} and Section~\ref{sec:ordinal-pivot}, respectively. We show the convergence of the algorithm in Section~\ref{sec:convergence}. In Section~\ref{sec:perturbation}, we show that polynomial-time convergence can also be proved if we apply Scarf's algorithm to a suitably non-degenerate perturbation of the bipartite matching polytope (recall that, in a non-degenerate polytope, the behaviour of Scarf's algorithm is uniquely determined).

\subsection{Basic Properties}

\subsubsection{Structure of $B$: Forest with Single Loops}\label{sec:basic-of-scarf}

Next two lemmas characterize the bases of the matching polytope. The proofs are fairly standard (see, e.g., the characterization of bases in network polytopes~\cite{bertsimas1997introduction}) and given for completeness in Appendix~\ref{sec:app:proof-basis}. \begin{lemma}\label{lem:forest}
Let $B$ be a set of $n$ columns of $A$ and $E_B=E_B^v\cup E_B^\ell$ be the set of corresponding edges of $G$, where $E_B^v$ consists of all valid edges and $E_B^\ell$ consists of all loops in $E_B$. $B$ is a (possibly infeasible) basis of $A$ if and only if
\begin{enumerate}
    \item $(V,E_B^v)$ is a forest.
    \item Each connected component 
    of the forest has exactly one loop in $E_B^\ell$.
\end{enumerate}
\end{lemma}

\begin{lemma}\label{lem:tree}
Consider any basis $B$ of $A$. Up to permuting rows and columns, the submatrix $A_B$ of $A$ corresponding to $B$ has the form 
\begin{displaymath}
A_B=\left(\begin{array}{cccc}
    B_1 & & &  \\
     & B_2 & &  \\
      & & \ddots & \\
     & & & B_\tau
\end{array}\right) \qquad 
\hbox{where } 
B_\omega=\left(\begin{array}{cccc}
    1 & * & * & * \\
     & 1 & \ddots & * \\
    &  & \ddots & * \\
     &  & & 1
\end{array}\right),
\end{displaymath}
i.e.,  $B_\omega$ is an upper-right matrix with $1$s on diagonal, for $\omega \in [\tau]$ (in the matrices above, no entry means $0$, and $*$ means either a $0$ or $1$). Moreover, in each matrix $B_\omega$, the first column corresponds to a loop, and each other column has exactly two entries equal to $1$. 
\end{lemma}

The structure presented in Lemma~\ref{lem:forest} is called \emph{forest with single loops}. An illustration is given in Figure~\ref{fig:bipartite-evencycle}.

\subsubsection{Structure of $D$: Almost-feasibility and Separator}\label{obsec}

Recall that all ordinal bases visited by the algorithm are almost-feasible, and $|D\cap B|=n-1$ for some feasible basis $B$ such that $(B,D)$ is a Scarf pair (see Section~\ref{sec:scarf:almost-feasible-etc} and Section~\ref{sec:scarf:degeneracy}). Also recall that for an almost-feasible basis $D$, we let be $E_D$ the set of edges corresponding to columns in $D$ (including loops). Note that, if $D$ is an almost-feasible ordinal basis, then there is a submatrix of $C_D$ of size $n\times (n-1)$ that has the structure from Lemma~\ref{lem:tree}, minus the first column of $B_1$ (which corresponds to the loop incident to node $m_1$, e.g., column 1). 

This section is devoted to the proof of the following proposition, containing properties of almost-feasible ordinal bases.

\begin{proposition}\label{prop:separator}
Let $D$ be an almost-feasible ordinal basis and $u$ the utility vector associated to it. If $u_1\in \mathcal{L}$ and $D$ contains at least one element in $\{2,\dots,k\}$, then there exists an index $i$, $2\le i\le k$, such that in the graph $G_D=(V,E_D)$: \par
\begin{enumerate}
\item  $m_i$ is incident to both a loop and one or more valid edges. \item $m_2,\dots,m_{i-1}$ are incident to valid edges only; $m_{i+1},\dots,m_{k}$ are incident to loops only. 
\item $m_1$ is not incident to any edge.
\end{enumerate}
Following the discussion in Section~\ref{sec:hl-poly}, we call $m_i$ the \emph{separator} of $D$.

\end{proposition}

An illustration of Proposition~\ref{prop:separator} is given in Figure~\ref{fig:bipartite-evencycle}. We remark that Proposition~\ref{prop:separator} and the related definition of separator apply only as long as $u_1 \in {\cal L}$. As we will see in Section~\ref{sec:convergence}, we will define $m_1$ to be the separator when $u_1 \in {\cal M}$.

\smallskip

It is useful to remark that we will frequently use the relation $B=R\cup\{1\}$, $D=R\cup\{j_t\}$. We start with some properties of the utility vector of an almost-feasible-basis.

\begin{lemma}\label{lem:uvectorlemma}
The utility vector of an almost-feasible ordinal basis $D$ satisfies 
$$u_i\in \mathcal{S}\cup \mathcal{M}, \forall i\neq 1.$$
\end{lemma}

\begin{proof}

Since $D$ has $n-1$ columns in common with a feasible basis $B$, and does not contain the first column, we have $B\setminus\{1\}\subset D$. By Lemma \ref{lem:tree}, any submatrix $A_B$ associated to a feasible basis $B$ has a special diagonal form. In particular, for each row $i\neq 1$, an edge incident to $v_i$ (possibly, a loop) appears among the columns of $R$. 
Hence, for every row $i$, except possibly the first, there is a column $j \in R$ such that $c_{i,j}\in \mathcal{S} \cup \mathcal{M}$. Thus the minimum element of those rows in $D$ must belong to $\mathcal{S}\cup \mathcal{M}$, which completes the proof.

\end{proof}

By Definition~\ref{def:disliked}, we can find a one-to-one correspondence of rows and columns in $D$. Recall that, intuitively, for each node $v_i\neq v_1$, the minimizer of row $i$ in $D$ corresponds to the worst choice between all edges incident to $v_i$ among the columns in $D$.

Clearly, if a loop $e_i$ satisfies $i\in D$, then $e_i$ is $v_i$-disliked since $c_{i,i}=0$ is the minimizer. Now consider a valid edge $e_j=(m,w)$. Column $c_j$ contains numbers in $\mathcal{L}$, except in rows $m$ and $w$. If $e_j\in E_D$, then as a corollary of Lemma~\ref{lem:uvectorlemma}, only $m$, $w$ or $m_1$ can dislike $e_j$:
\begin{corollary}\label{obs:almost-feasible-v-i}
Let $D$ be an almost-feasible basis. For any valid edge $e_j=(m,w)$, if $e_j\in E_D$, then $e_j$ can only be $m$-disliked, $w$-disliked, or $m_1$-disliked, and it is exactly one of the three.
\end{corollary}

\begin{lemma}\label{unmatchedwoman}
Let $D$ be almost-feasible basis and $B$ a feasible basis associated to it. If $u_1\in \cal{L}$, then there exists $\bar w \in W$ such that $(\bar{w},\bar{w}) \in E_B \cap E_D$ and $\bar w$ is not properly matched in $\mu_B$.
\end{lemma}
\begin{proof}

Let $B$ be the associated feasible basis such that
$D=B\cup\{j_t\}\setminus\{1\}$ as in Definition \ref{almostfeasibledef}. Then $R\subset D$. $B$ corresponds to a feasible matching $\mu_B$. Since $u_1\in \mathcal{L}$, no edge is incident to $m_1$ in $E_D$, thus no valid edge is incident to $m_1$ in $E_B$, which implies $m_1$ is not properly matched in $\mu_B$. 
Since $|M|=|W|$, there exists some woman $\bar{w}$ who is also not properly matched. Thus $x_{(\bar{w},\bar{w})}=1$, i.e.,  the loop $(\bar{w},\bar{w})\in E_B\cap E_D$.\end{proof}

Next lemma shows a fundamental property en route to the proof of Proposition~\ref{prop:separator}. 

\begin{lemma}\label{lem:man-order}
Consider any almost-feasible ordinal basis $D$, and suppose $u_1 \in \mathcal{L}$. Then: 
\item[(i)] If $i\notin D$ ($2\le i\le k$) and $\ell\leq i$, then $\ell\notin D$.
\item[(ii)] If there exists $2\le i\le k$ such that $i-1\notin D$ and $i\in D$, then the rightmost column in $D$ corresponds to a valid edge $(m_i,w)$ with some $w\in W$.
\end{lemma}
\begin{proof}

(i) Since $D$ is almost-feasible, we always have $1\notin D$. The case $i=2$ is trivial. Consider $3\le i\le k$ and $i\notin D$. Assume by contradiction that there exists $2\le \ell<i$ such that $\ell\in D$.

Let $\bar{w}$ be the woman not properly matched whose existence is guaranteed by Lemma \ref{unmatchedwoman}. Consider the edge $(m_\ell,\bar{w})$. We show $(m_\ell,\bar{w})\notin E_D$. Assume by contradiction $(m_\ell,\bar{w})\in E_D$. Then $(m_\ell,\bar{w})$ is neither $m_\ell$-disliked nor $\bar{w}$-disliked, since $(m_\ell,m_\ell),(\bar{w},\bar{w})\in E_D$. By Corollary \ref{obs:almost-feasible-v-i}, $(m_\ell,\bar{w})$ is $m_1$-disliked. $(m_\ell,\bar{w})$ is the rightmost column in $D$, since by definition of $C$, any column $j$ to the right of $(m_\ell,\bar{w})$ satisfies $c_{1 j}<c_{1 (m_\ell,\bar{w})}$, contradicting the fact that $u_1=c_{1 (m_\ell,\bar{w})}$.  Using the same argument, since any valid edge incident to $m_i$ corresponds to a column on the right of $(m_\ell,\bar{w})$ (recall $i>\ell$), no valid edge in $E_D$ is incident to $m_i$. Moreover, we know $(m_i,m_i)\notin E_D$ by hypothesis. Hence, no edge is incident to $m_i$ in $B$, either, a contradiction.

Hence, $(m_\ell,\bar{w})\notin E_D$. Since $D$ is an ordinal basis, 
there exists $a$ such that $u_{a}>c_{a,(m_\ell,\bar{w})}$. Notice that column $(m_\ell,\bar w)$ has entries in $\mathcal{L}$, except for rows $m_\ell,\bar w$, that have entries in $\mathcal{M}$. Hence, 
by Lemma \ref{lem:uvectorlemma}, we have $a \in \{m_1,m_\ell,\bar w\}$. Since $(m_\ell,m_\ell), (\bar w,\bar w) \in E_D$, we have $u_{(m_\ell,m_\ell)}=u_{(\bar w,\bar w)}=0$. Hence, $a=m_1$. Recall from above that $u_1$ is realized at the rightmost column of $D$. Since, by hypothesis $i \notin D$ and $\ell < i$ and by construction all edges incident to node $m_i$ follow all columns from $D$, 
we have $u_1\le c_{(m_i,\mu_B(m_i))}< c_{(m_\ell,\bar w)}$, obtaining the required contradiction.

\medskip
\noindent (ii) By (i), $1,\dots,i-1\notin D$ and $i,i+1,\dots,k\in D$.
Now consider the $m_1$-disliked edge $e=(m_\ell,w)$ in $D$. Since $u_1 \in \mathcal{L}$, $e$ is not incident to $m_1$ and it is therefore the rightmost column in $D$.

If $\ell\le i-1$, then $e$ is also $m_\ell$-disliked because when $\ell\notin D$, the unique $m_\ell$-disliked edge occurs at the rightmost entry from the set of columns incident to $m_\ell$, 
which is exactly $e$. Since $m_\ell\neq m_1$, we contradict the bijection from Definition~\ref{def:disliked}.

If $\ell\ge i+1$, then consider the edge $(m_i,\bar{w})$. If $(m_i,\bar{w})\in E_D$ then neither $m_i$ (since $(m_i,m_i) \in E_D$), nor $\bar w$ (since $(\bar{w},\bar{w}) \in E_D$), nor $m_1$ (since $e \in E_D$) dislike $(m_i,\bar{w})$. This contradicts Corollary~\ref{obs:almost-feasible-v-i}. If $(m_\ell,\bar{w})\notin E_D$, then the corresponding column $c_{(m_\ell,\bar{w})}$ is strictly greater than $u$. Both are contradictions.\par
Therefore $\ell=i$ and $e=(m_i,w)$ is the rightmost column for some $w\in W$. \end{proof}

We now prove Proposition~\ref{prop:separator}. Consider the graph $G_D=(V,E_D)$. Let $2 \leq i \leq k$ such that $i-1\notin D$ and $i\in D$. Since $u_1 \in \mathcal{L}$, $m_1$ is not incident to any edge, proving 3. $i \in D$ implies that $m_i$ is incident to a loop, while Lemma~\ref{lem:man-order}, part (ii) implies that $m_i$ is incident in $G_D$ to one valid edge $(m_i,w)$. This proves 1. By Lemma~\ref{lem:man-order}, part (i), $m_2,\dots,m_{i-1}$, are not incident to loops, but each of them must be incident to at least one valid edge by the definition of almost-feasibility and Lemma~\ref{lem:tree}. On the other hand, $m_{i+1},\dots,m_k$ are only incident to loops, by Lemma \ref{lem:man-order}(i). This shows 2 and concludes the proof of Proposition~\ref{prop:separator}.

\subsection{Pivoting}\label{sec:pivoting}

In the next two sections, we will discuss how Scarf's algorithm performs a generic iteration, as long as $u_1 \in {\cal L}$. As we will see later, the behavior of the algorithm is also similar when $u_1 \in {\cal M}$, but the arguments are slightly different.

Suppose we have a Scarf pair $(B,D)$, with $x$ being the basic feasible solution associated to $B$ and $u$ the utility vector associated to $D$. We let
\begin{equation}\label{eq:BDR}
B=\{1,j_1,\dots,j_{t-1},j_{t+1},\dots,j_n\}, \, D=\{j_1,\dots,j_t,\dots,j_n\}, \, 
R=B \cap D = B\setminus \{1\}.
\end{equation}
Following Proposition~\ref{prop:separator}, we let $m_i$ be the current separator. Recall that an iteration starts with a cardinal pivot such that $j_t$ enters $B$ and some $j_\ell\neq j_t$ leaves $B$, leading to a new feasible basis $B'=\{1,j_1,\dots,j_{\ell-1},j_{\ell+1},\dots,j_n\}$ associated to the basic feasible solution $x'$.
If $j_\ell\neq 1$, this iteration continues with an ordinal pivot such that $j_\ell$ leaves $D$ and some $j^*\neq j_\ell$ enters $D$. We have therefore a new ordinal basis $D'=\{j^*,j_1,\dots,j_{\ell-1},j_{\ell+1},\dots,j_n\}$, to which a new utility vector $u'$ is associated. 

If $j^*\neq 1$, one iteration ends and we continue to the next. 

\subsubsection{Cardinal Pivots}\label{sec:cardinal-pivot}
In this section, we discuss how Scarf's algorithm performs a cardinal pivot.
Since $j_t$ is given, the basic feasible solution obtained after a cardinal pivot is performed is uniquely determined (see, e.g.,~\cite[Section 3.2]{bertsimas1997introduction}), and we denote it by $x'$. On the other hand, because of degeneracy, there may be multiple indices $j_\ell$ that can leave the basis, hence multiple basis corresponding to $x'$. So our task is to find an appropriate $j_\ell$. In particular, we show the following.


\begin{lemma}\label{lem:cardpivot}
Consider a cardinal pivot of Scarf's algorithm, where $u_1\in \mathcal{L}$ and the separator is $m_i$. If $j_t$ is a man-disliked (w.r.t.~$D$) valid edge, then we can always let the leaving column $j_\ell$ be either the loop $e_i$ or some woman-disliked edge (w.r.t.~$D$).
\end{lemma}

\begin{proof}
    
We show that there is a basis $B'$ corresponding to $x'$ of the form $B \cup \{j_t\} \setminus \{j_\ell\}$, where $j_\ell$ is either the loop $(m_i,m_i)$ or a woman-disliked edge. Note that it may be that $x=x'$ and / or that there are multiple basis corresponding to $x'$.

By Lemma \ref{lem:forest}, $G_B$ is a forest with single loops. When a valid edge $e_{j_t}$ is added to our graph, one of the following happens:

\begin{enumerate}

\item[(I)]$e_{j_t}$ joins two different trees $T_1,T_2$ of $(V,E_B^v)$, as to form a larger tree $T$ with two loops.

\item[(II)] $e_{j_t}$ connects two nodes of a same tree $T_1$ of $(V,E_B^v)$.

\end{enumerate}

Suppose (I) happens and consider the path $P$ connecting the two loops of $T$.  Since $u_1\in \mathcal{L}$, we know that there is no valid edge incident to $m_1$ in $E_D$, thus $P$ is not incident to $m_1$. Therefore, all edges of $P$ are contained in $E_D$.  
\begin{claim}\label{cl:P-starts-with-loop}
$P$ starts at $m_i$ with a loop and ends at a woman $\bar{w}$ with a loop. Moreover, suppose $P$ is incident to $p$ nodes, with
\begin{equation}\label{eq:alter-disliked-path}
    P=(m_{i_1},m_{i_1}),(m_{i_1},w_{i_1}),(w_{i_1},m_{i_2}),\dots,(m_{i_{\frac{p}{2}}},w_{i_{\frac{p}{2}}}),(w_{i_{\frac{p}{2}}},w_{i_{\frac{p}{2}}})
\end{equation}
such that $m_{i_1}=m_i$ and $w_{i_{\frac{p}{2}}}=\bar{w}$. Then the edges of $P$ are disliked by
$$m_{i_1},m_1,w_{i_1},m_{i_2},w_{i_2},\dots,m_{i_{\frac{p}{2}-1}},w_{i_{\frac{p}{2}-1}},m_{i_{\frac{p}{2}}},w_{i_{\frac{p}{2}}}$$
in this order. 
\end{claim}

\noindent \emph{\underline{Proof of Claim~\ref{cl:P-starts-with-loop}.}} Suppose $P$ is incident to $p$ nodes and has therefore $p-1$ valid edges, and $p+1$ edges in total (including loops). By Definition~\ref{def:disliked} and Corollary \ref{obs:almost-feasible-v-i}, each edge of $P$ is disliked by exactly one node from $V_P\cup\{m_1\}$. In particular, the $m_1$-disliked edge is on $P$, and by Proposition~\ref{prop:separator} it is exactly the rightmost column $(m_i,w) \in E_D$, where $m_i$ is the separator and $w$ is some woman. Hence, one of the loops on $P$ is $(m_i,m_i)$, and the first node of $P$ can without loss of generality be assumed to be $m_i$. On the other hand, again by Proposition~\ref{prop:separator}, the last node of $P$ is a woman $\bar{w}$.


We have therefore 
obtained (\ref{eq:alter-disliked-path}), where $m_{i_1}=m_i$, $w_{i_1}=w$ and $w_{i_{\frac{p}{2}}}=\bar{w}$.
It is clear that $(m_{i_1},m_{i_1})$ is $m_{i_1}$-disliked, $(m_{i_1},w_{i_1})$ is $m_1$-disliked, and $(w_{i_{\frac{p}{2}}},w_{i_{\frac{p}{2}}})$ is $w_{i_{\frac{p}{2}}}$-disliked. Then 
the edges of $P$ are disliked by
$$m_{i_1},m_1,w_{i_1},m_{i_2},w_{i_2},\dots,m_{i_{\frac{p}{2}-1}},w_{i_{\frac{p}{2}-1}},m_{i_{\frac{p}{2}}},w_{i_{\frac{p}{2}}}$$
in this order. \hfill $\small \blacksquare$

\smallskip

Recall that we are letting some man-disliked valid edge $e_{j_t}$ enter the basis, with $e_{j_t} \in P$. Since $j_t\notin B$, we have $x_{j_t}=0$. 

\smallskip

\emph{Case 1: $x_{j_t}'=1$.} We claim that there exists an $x$-augmenting path $P^A$ starting at $m_i$ with a loop, such that $e_{j_t}\in P^A\subset E_D.$
In fact, consider the matchings $\mu_x$ and $\mu_{x'}$. We have $e_{j_t}\notin \mu_x$ and $e_{j_t}\in \mu_{x'}$. Now define the edge set $E_{change}=\{e\in E_D|x_e\neq x_e'\}$. Then $e_{j_t}\in E_{change}$. Denote by $P^A$ the connected component in $(V,E_{change})$ that contains $e_{j_t}$. Notice that every edge in $P^A$ still differs in $x$ and $x'$, then for any $e\in P^A$, one of $x_e$, $x_e'$ takes value $1$ and the other takes value $0$. 

Consider any node $v$ that belongs to $P^A$. There are exactly two edges in $P^A$ (including loops) incident to $v$, which are precisely the edges $v$ is incident to in $\mu_x$ and $\mu_{x'}$ (notice that, by the definition of matching presented in Section~\ref{sec:notation}, every node is incident to exactly one edge in $\mu_x$ and $\mu_{x'}$). Since there is no cycle in $E_D$ and $P^A$ is connected, $P^A$ can only be a path. This path $P^A$ is $x$-alternating because any consecutive two edges with $x$-value $1,0$ will cause the infeasibility of $x$, $x'$, respectively. Furthermore, $P^A$ is $x$-augmenting since, by maximality, its starting and ending edges can only be loops. Hence $P^A$ is the desired $x$-augmenting path that contains $e_{j_t}$.
By the definition of $x$-augmenting path, $P^A$ consists of at least two loops as the endpoints. By the structure of $E_D$, there is only one path in $E_D$ that contains more than one loop, which is $P$. Therefore, $P^A=P$, which implies that $P$ starts at $m_i$ with a loop because of Claim~\ref{cl:P-starts-with-loop}.

We are left to show that $B\cup \{j_t\} \setminus \{i\}$ is a feasible basis whose associated vertex is $x'$. 
We first prove that $B'=B\cup\{j_t\}\setminus\{i\}$ is a linearly independent set. We can observe this in the graph $G_{B'}=(V,E_{B'})$. 
Recall that, by Lemma \ref{lem:forest}, $G_B$ has a forest with single loops structure. Adding $j_t$ and removing $i$ from $B$ keeps the structure, and each component of $G_{B'}$ has exactly one loop. Hence, using again Lemma \ref{lem:forest}, $B \cup \{j_t\} \setminus \{i\}$ is a basis. In order to conclude that it corresponds to $x'$, observe that the support of $x'$ is contained in $D \cup \{1,j_t\}\setminus \{i\}\subset  B \cup \{j_t\} \setminus \{i\}$.

\smallskip

\emph{Case 2: $x_{j_t}'=0$.} Hence, we are in a degenerate pivot and $x'=x$. We claim that there is no $x$-augmenting path $P^A$ that contains $e_{j_t}$ in $E_D$. Otherwise, if such $P^A$ exists, define $y\in \R^m$ such that
\begin{displaymath}
y_e=\left\{
\begin{array}{cc}
  1-x_e,   & \textrm{ if } e\in E_{P^A},\\
  x_e,   & \textrm{ if } e\notin E_{P^A}.
\end{array}
\right.
\end{displaymath}
Since $P^A$ is $x$-augmenting, we have $Ay=b$ and $y_{j_t}=1-x_{j_t}=1$. Let $v_j$ be one of the endpoints of $P^A$, then $y_{(v_j,v_j)}=1-x_{(v_j,v_j)}=0$. Define $B^{(y)}=B\cup\{j_t\}\setminus\{j\}$. Similarly to Case 1, we can deduce that $B^{(y)}$ is a basis of $A$, and by definition $B^{(y)}y_{B^{(y)}}=b$. Thus $y$ is a basic feasible solution, with $y_{j_t}=1$, a contradiction. 

Notice that $E_P\subset E_D$, and $P$ is not $x$-augmenting. Then by Claim~\ref{cl:P-starts-with-loop} there exists at least one edge $e\in P$, which is either $m_{i_1}$-disliked or woman-disliked such that $x_{e}'=x_{e}=0$. Following an argument similar to Case 1, we can let the edge $e_{j_\ell}=e$ leave the basis as to obtain the basis $B'=B\cup\{j_t\}\setminus\{j_\ell\}$ associated to $x'=x$.

\smallskip

If (II) happens, then the entering valid edge $e_{j_t}$ creates a cycle $Q$ in $T$, where $Q$ must be even. $Q$ may be incident to at most one loop. By Lemma~\ref{lem:forest}, the pivoting lets one of the valid edges in $Q$ exit the basis, in order to form $B'$. Suppose $Q$ has $p$ (valid) edges  incident to $p$ nodes, and recall that  $p$ is even. 

\begin{claim}\label{cl:every-other-w-disliked}
Every other edge of $Q$ is woman-disliked.
\end{claim}
\smallskip
\noindent \emph{\underline{Proof of Claim~\ref{cl:every-other-w-disliked}.}}
\emph{Case 1: $Q$ contains $m_i$.} We claim that $(m_i,w)$ belongs to $Q$. If $(m_i,w)$ does not belong to $Q$, then by Corollary~\ref{obs:almost-feasible-v-i}, the edges of $Q$ have to be disliked by nodes of $Q\setminus \{m_i\}$ (for $m_i$ dislikes $(m_i,m_i)$), contradicting Definition~\ref{def:disliked}. Now let
\begin{equation}\label{eq:alter-disliked-cycle}
    Q=(m_{i_1},w_{i_1}),(w_{i_1},m_{i_2}),\dots,(m_{i_{\frac{p}{2}}},w_{i_{\frac{p}{2}}}),(w_{i_{\frac{p}{2}}},m_{i_1})
\end{equation}

where $m_{i_1}=m_i$ and $w_{i_1}=w$. Since $(m_{i_1},w_{i_1})$ is $m_1$-disliked, then the edges on $Q$ are disliked by
$$m_1,w_{i_1},m_{i_2},w_{i_2},\dots,m_{i_{\frac{p}{2}-1}},w_{i_{\frac{p}{2}-1}},m_{i_{\frac{p}{2}}},w_{i_{\frac{p}{2}}}$$
in this order. 

\smallskip


\emph{Case 2: $Q$ does not contain $m_i$.} Let
$$Q=(m_{i_1},w_{i_1}),(w_{i_1},m_{i_2}),\dots,(m_{i_{\frac{p}{2}}},w_{i_{\frac{p}{2}}}),(w_{i_{\frac{p}{2}}},m_{i_1}).$$
Since the unique $m_1$-disliked edge does not belong to $Q$ (otherwise $m_i$ is incident to $Q$ and we are in case $1$), the edges on $Q$ must be disliked by man, woman alternatively.  \hfill $\small \blacksquare$

\smallskip

By a similar argument as in (I), we can obtain the desired result. In detail, if all woman-disliked edges on $Q$ have $x$-value $1$, then $Q$ is $x$-alternating, with all woman-disliked edges on $Q$ having  $x$-value $1$, and all man-disliked edges having $x$-value $0$. 
Define $y\in \R^m$ as 
\begin{displaymath}
y_e=\left\{
\begin{array}{cc}
  1-x_e,   & \textrm{ if } e\in E_Q,\\
  x_e,   & \textrm{ if } e\notin E_Q.
\end{array}
\right.
\end{displaymath}
It is not difficult to check that $y$ is a basic feasible solution with $y_{j_t}=1$, and $B'=B \cup \{j_t\}\setminus \{j_\ell\}$ is a basis corresponding to $y$, where $j_\ell$ is chosen to be any woman-disliked edge of $Q$. If conversely there is a woman-disliked edge $e$ with $x_e=0$, then we have $x=x'$ and let $e_{j_\ell}=e$, $B'=B \cup \{j_t\}\setminus \{j_\ell\}$, and the claim follows analogously.

\end{proof}

\begin{remark}
In the proof of Lemma \ref{lem:cardpivot}, we construct a vector $y$ several times, as the symmetric difference of the matching $\mu_x$ and the alternating subgraph ($P$ or $Q$). Formally, $y$ is the characteristic vector of $\mu_x\Delta E_P$ (resp., $\mu_x\Delta E_Q)$. In the theory of bipartite matching, it is well-known that a matching $\mu$ is maximal if and only if no $\mu$-augmenting path exists. We can see the similarity in our redefined matching problem.
\end{remark}

\subsubsection{Ordinal Pivots}\label{sec:ordinal-pivot}
This is the continuation of the previous section. We discuss how Scarf's algorithm finds an entering column $j^*$ when $j_\ell$ leaves $D$, as to form $D'$. We still assume $u_1\in \cal{L}$ and that $m_i$ is the separator in $D$.\par

If $e_{j_\ell}$ is $v_{i_\ell}$-disliked in $D$, then according to Definition~\ref{def:appopdef}, column $j_r$ satisfies that $c_{i_\ell,j_r}$ is the second least element on row $i_\ell$ in $C_D$. Recall that we call $j_r$ (resp., $e_{j_r}$) a \emph{reference column} (resp., \emph{edge}). Suppose the reference edge $e_{j_r}$ is $v_{i_r}$-disliked w.r.t.~$D$. Let $j^*$ be the entering column in the ordinal pivot. The following observation follows easily from the last definition.

\begin{lemma}\label{lem:who-dislikes-who}
In the ordinal basis $D'=D \setminus\{j_\ell\}\cup\{j^*\}$, we have that $e_{j_r}$ is $v_{i_\ell}$-disliked and $e_{j^*}$ is $v_{i_r}$-disliked.
\end{lemma}

Using the notations above, the following property plays a key role in the ordinal pivot:
\begin{lemma}\label{lem:referinM}
$c_{i_\ell,j_r}\in\cal{M}$.
\end{lemma}

\begin{proof}

We can claim this by showing $c_{i_\ell,j_r}\notin \mathcal{S}\cup \mathcal{L}\cup \mathcal{XL}$.\par
If $c_{i_\ell,j_r}=0$, then $c_{i_\ell,j_r}<c_{i_\ell,j_r}=0$, a contradiction. Thus $c_{i_\ell,j_r}\notin \mathcal{S}$.\par
If $c_{i_\ell,j_r}\in \cal{L}\cup \cal{XL}$, then $c_{i_\ell,j_r}\ge k+1$. Since $c_{i_r,j_\ell}$ is the second least element from $\{c_{i_r,j} : j \in D\}$, there is at most one edge in $E_D$ incident to $v_{i_r}$, which, if exists, must be the leaving edge $e_{j_\ell}$. From the discussion in Section~\ref{sec:pivoting}, $B'=D\cup\{1\}\setminus\{j_\ell\}$ is a basis visited by Scarf algorithm. As a result, however, there is no edge incident to $v_{i_r}$ in $E_{B'}$, which contradicts the feasibility of $B'$. Thus $c_{i_\ell,j_r}\in \cal{M}$. 
    
\end{proof}
\smallskip

From {Remark}~\ref{lem:who-dislikes-who} and Lemma~\ref{lem:referinM} we can immediately deduce that $e_{j_r}$ is not a loop. 
\begin{corollary}\label{referlemma}
In any ordinal pivot, the reference edge $e_{j_r}$ is a valid edge.
\end{corollary}
The ordinal pivot will change the utility vector $u$, but only in components $u_{i_\ell}$ and $u_{i_r}$. The following lemma translates Lemma~\ref{lem:app:changeofu}, i.e.,~some basic facts on the mechanics of Scarf's algorithm, into graphic language.
\begin{lemma}\label{lem:utilitychange}
The ordinal pivot gives $(D,u)\to(D',u')$, where $u_{i_\ell}'>u_{i_\ell}$, 
$u_{i_r}'<u_{i_r}$, and $u_{i}'=u_{i},\textrm{ for $i\neq i_\ell,i_r$}$.
Correspondingly, we have that:
\begin{enumerate}
    \item $v_{i_\ell}$ dislikes the leaving edge $e_{j_\ell}$ in $D$ and dislikes the reference edge $e_{j_r}$ in ${D'}$.
    \item $v_{i_r}$ dislikes the reference edge $e_{j_r}$ in $D$ and dislikes the entering edge $e_{j^*}$ in ${D'}$.
    \item Any other node dislikes the same edge in $D$ and ${D'}$.
\end{enumerate}
\end{lemma}

As a continuation of Lemma \ref{lem:cardpivot}, we will see that a leaving loop $e_i$ or a leaving woman-disliked edge will always return a new man-disliked valid edge $e_{j^*}$ in an ordinal pivot.
\begin{lemma}\label{lem:wdisordpivot}
Assume $u_1\in \mathcal{L}$, $\{2,\dots, k\} \cap D \neq \emptyset$, and that the leaving edge $e_{j_\ell}$ is woman-disliked (w.r.t.~$D$). Then:
\begin{enumerate}
\item  $v_{i_\ell}$ is a woman and $v_{i_r}$ is a man;
\item $e_{j^*}$ is man-disliked in $D'$; \item the woman $v_{i_\ell}$'s utility increases, i.e., $u'_{i_\ell}>u_{i_\ell}$, and any other woman's utility does not change.
\end{enumerate}
\end{lemma}
\begin{proof}

Since $e_{j_\ell}$ is woman-disliked and $v_{i_\ell}$-disliked (w.r.t.~$D$), $v_{i_\ell}$ can only be a woman. Since by Corollary~\ref{referlemma}, the reference edge $e_{j_r}$ is a valid edge and is not disliked by its incident woman $v_{i_\ell}$, then it is disliked by a man, by Corollary~\ref{obs:almost-feasible-v-i}. Therefore, $v_{i_r}$ is a man. This shows 1. 

The ordinal pivot introduces a new $v_{i_r}$-disliked edge $e_{j^*}$ into $D'$. Recall that  $m_i$ is the separator in $D$. Assume by contradiction that $e_{j^*}$ is a loop. Then $e_{j^*}=(v_{i_r},v_{i_r})$. By Lemma \ref{lem:man-order}, the new ordinal basis $D'$ must have a separator. We claim that this separator can only be $m_{i-1}$. Indeed, loops corresponding to $e_j$ for $j\geq i$ belong to $D'$. Hence, one of $m_1,\dots, m_{i-1}$ is the separator in $D'$. However, if $m_j$ is the separator for some $j\leq i-2$, then we contradict Lemma~\ref{lem:man-order}(i). Hence, $m_{i-1}$ is the separator.

By definition, all valid edges incident in $D$ to $m_i$ must leave $D$. By Lemma~\ref{lem:man-order}(ii), one such edge is $m_1$-disliked. 
This contradicts the fact that $e_{j_\ell}$ is woman-disliked. Therefore, $e_{j^*}$ is a valid edge. This shows 2. \par
As for the change of utility vector, we have by Lemma~\ref{lem:utilitychange}, part 1,
$$u_{i_\ell}'=c_{i_\ell,j_r}>c_{i_\ell,j_r}=u_{i_\ell}.$$
Since $v_{i_r}$ is a man, it follows from Lemma~\ref{lem:utilitychange}, part 3 that every other woman's utility does not change. This shows 3. 
\end{proof}

We remark that, for the following fundamental lemma to holds, we need the extra properties of matrix $C$ that distinguish it from the generic matrix defined by Bir\'o and Fleiner~\cite{biro2016matching}, as discussed in Section~\ref{sec:matrixdesign}.

\begin{lemma}\label{lem:mloopordpivot}
Assume $u_1\in \mathcal{L}$ and that the leaving edge $e_{j_\ell}$ is the loop $(m_i,m_i)$.  Then:
\begin{enumerate}
    \item  $v_{i_\ell}=m_i$, $v_{i_r}=m_1$;
    \item Let $u'$ be the utility vector of the ordinal basis $D'$ obtained wal pivot. Then a new $m_1$-disliked (w.r.t.~$D'$) valid edge $e_{j^*}$ enters and:
    \begin{enumerate}
\item[(i)] If $2\le i <k$, $e_{j^*}$ is incident to $m_{i+1}$, and 
$$u_1'<u_1,\textrm{ but still $u_1'\in \mathcal{L}$}.$$
$$u_{i}'=u_i,\textrm{ for $k<i\le 2k$}.$$
\item[(ii)] If $i=k$, $e_{j^*}$ is incident to $m_1$, and
$$u_1'<u_1,\textrm{ and $u_1'\in \mathcal{M}$}.$$
$$u_{i}'=u_i,\textrm{ for $k<i\le 2k$}.$$
\end{enumerate}
\end{enumerate}
\end{lemma}
\begin{proof}

Loop $(m_i,m_i)$ is $m_i$-disliked w.r.t.~$D$, thus $v_{i_\ell}=m_i$. By Lemma~\ref{lem:referinM}, $c_{i_\ell,j_r}\in\cal{M}$, thus the edge $e_{j_r}$ is incident to $m_i$. By Lemma \ref{lem:man-order}(ii), the rightmost column in $D$ is also incident to $m_i$. Moreover, the rightmost column in $D$ is exactly $j_r$ because $m_i$'s second worst choice $e_{j_r}$ between all edges is the worst choice between all valid edges, which corresponds to the rightmost column. Since $u_1\in \mathcal{L}$, the rightmost column $j_r$ is $m_1$-disliked, hence $v_{i_r}=m_1$. This shows 1. \par
Notice that by Lemma \ref{unmatchedwoman}, there is a woman $\bar{w}$ whose loop $(\bar{w},\bar{w})\in E_D$. We use this fact to prove part 2..\par
(i) If $2\le i<k$, then by Definition \ref{def:appopdef}, the entering column $j^*$ satisfies 
\begin{equation}\label{lem:mloopordpivoteq1}
c_{hj^*}>\bar{u}_\ell\textrm{ for all $h\neq 1$}.
\end{equation}
Consider the edge $e_j=(m_{i+1},\bar{w})$. Column $c_j$ satisfies the above condition and $j\notin D$ (By Lemma \ref{lem:man-order}(ii), the rightmost column in $D$ is incident to $m_i$, thus no valid edge incident to $m_{i+1}$ belongs to $D$). Therefore,
\begin{equation}\label{lem:mloopordpivoteq2}
u_1'=c_{1j^*}\ge c_{1,(m_{i+1},\bar{w})}.
\end{equation}
Thus $u_1'\in \mathcal{L}$. Notice that $u_1'<u_1$ must hold, otherwise $c_{j^*}>u$, a contradiction. Suppose $e_{j_r}=(m_i,w)$, then
\begin{equation}\label{lem:mloopordpivoteq3}
u_1'<u_1=c_{1,(m_i,w)}.
\end{equation}
(\ref{lem:mloopordpivoteq2}) and (\ref{lem:mloopordpivoteq3}) imply that edge $e_{j^*}$ is incident to either $m_i$ or $m_{i+1}$. Suppose the former happens, then by (\ref{lem:mloopordpivoteq3}), $c_{1,j^*}<c_{1,(m_i,w)}=c_{1,j_r}$, thus column $j^*$ is on the right of column $j_r$ in $C$, which implies 
$$c_{i,j^*}<c_{i,j_r}=\bar{u}_i.$$
This contradicts (\ref{lem:mloopordpivoteq1}) for $\ell=i$. Therefore, $e_{j^*}$ must be incident to $m_{i+1}$.\par
This ordinal pivot only changes the row minimizer of $m_i$ and $m_1$. Since no row minimizer corresponding to a women changes, we have
$$u_i'=u_i,\textrm{ for $k<i\le 2k$}.$$
(ii) If $i=k$, then consider the edge $e_j=(m_1,\bar{w})$. With a similar argument as in (i) we argue that $e_{j^*}$ can only be incident to $m_1$. Now suppose $e_{j^*}=(m_1,w')$. In this stage, a valid edge incident to $m_1$ first enters ordinal basis. Notice that
$$u_1'=c_{1,j^*}=c_{1,(m_1,w')}\in \mathcal{M}.$$
Also, the utility of each woman does not change, which completes the proof.
\end{proof}

\subsection{Convergence}\label{sec:convergence}

Recall that an ordinal pivoting is uniquely defined, once that a column entering the current ordinal basis has been selected. The cardinal pivoting rule described in Algorithm~\ref{alg:pivoting} will, in a polynomial number of iterations, lead to the convergence of Scarf's algorithm.

Recall that, at the first iteration of Scarf's algorithm, we have $u_i \in {\mathcal L}$.

\begin{algorithm}
\caption{Cardinal pivoting rule}\label{alg:cap}\label{alg:pivoting}
\begin{algorithmic}
\State Let $B$ be the current feasible basis, $D$ be the current ordinal basis with utility vector $u$, and $j_t$ the man-disliked (w.r.t.~$D$) edge that is going to enter $B$. 
\If {$u_1 \in {\cal L}$}
\State {Let $m_i$ be the separator.} 
\Else 
\State{Set $i=1$.} 
\EndIf
\If {$e_i=(v_i,v_i)$ is a candidate to leave the basis}
\State {Let $e_i$ leave the basis.}
\Else 
\State {Choose any woman-disliked valid edges (w.r.t.~$D$) that is a candidate to leave the basis as the variable that leaves the basis.}
\EndIf
\end{algorithmic}
\end{algorithm}

To prove polynomial-time convergence, we start with an auxiliary lemma. 

\begin{lemma}\label{lem:u1-in-M}
If Scarf's algorithm iteratively applies Algorithm~\ref{alg:pivoting} to perform a cardinal pivot while $u_1 \in {\cal L}$, then we obtain an ordinal basis with $u_1 \in {\cal M}$ after $O(k^2)$ steps. 
More in detail, at every step, we weakly increase the subscript of the separator (where we identify $k+1$ with $1$) and the total utility of women $~\sum_{w \in W}u_w$, and strictly increase at least one of them.
\end{lemma}

\begin{proof}

Suppose $k\ge 2$, else the statement is trivial. By construction, $B_0=\{1,2,\dots,n\}$ and $D_0=\{3k+1,2,3,\dots,n\}$, since $c_{1,3k+1}$ is the maximum entry in the first row outside the first $n$ columns, (see Example~\ref{ex:bi-Cmatrix} for an illustration). 

We first show by induction on the number of iterations that while $u_1 \in {\cal L}$, the input to Algorithm~\ref{alg:pivoting} is well-defined, and moreover, there is always a variable candidate to leave the basis that is either of the form $e_i$ (where $m_i$ is the current separator) or a woman-disliked edge (w.r.t.~the current ordinal basis $D$). Recall that an iteration is defined as the change of both an ordinal and a feasible basis.

For the basic step, note that $u_1=c_{1,3k+1}\in \mathcal{L}$, $1\notin D$ and $2\in D$. By Lemma \ref{lem:man-order}(ii) $m_2$ is the separator. At the first iteration, we execute a cardinal pivot to let $e_{3k+1}$ enter $B_0$, which as argued above is man-disliked. Hence, the input to Algorithm~\ref{alg:pivoting} is well-defined. The second part of the statement follows from Lemma~\ref{lem:cardpivot}.

For the inductive step, let us investigate the generic iteration $g$-th iteration, with $g \in \mathbb{N}$, of the algorithm. By inductive hypothesis, the input to Algorithm~\ref{alg:pivoting} in iteration $g-1$ is well-defined, and the leaving variable is chosen to be either $e_i$ (where $m_i$ is the current separator) or a woman-disliked edge (w.r.t.~$D$). We can then apply Lemma~\ref{lem:wdisordpivot} or Lemma~\ref{lem:mloopordpivot} to conclude that the edge entering the current feasible basis $B$ is man-disliked. Since every almost-feasible basis has a separator (see Proposition~\ref{prop:separator}), the input to Algorithm~\ref{alg:pivoting} is well-defined. The second part of the statement follows from Lemma~\ref{lem:cardpivot}.

Hence, Scarf's algorithm that iteratively applies Algorithm~\ref{alg:pivoting} for choosing a cardinal pivoting rule is well-defined. Let us now argue about its convergence to an ordinal basis with $u_1 \in {\cal M}$. Note that during an ordinal pivot, either we move the separator from $i$ to $i+1$ (when $i=n$, define $i+1=1$), while all women's utility stay constant (Lemma~\ref{lem:mloopordpivot}, part 2), or the separator does not move, but a woman's utility increases, while all other stay constant (Lemma~\ref{lem:wdisordpivot}, part 3). Throughout the algorithm, the utility of a woman is contained in ${\cal S} \cup {\cal M}$. Hence, after $O(k^2)$ iterations, we must have that the separator becomes $m_1$, which implies $u_1 \in {\cal M}$ (see Lemma~\ref{lem:mloopordpivot}, part 2). 

\end{proof}

Consider the first iteration when $u_1 \notin {\cal L}$. Then $u_1 \in {\cal M}$ and, before that, no valid edge incident to $m_1$ occurs in the intermediate feasible basis $B$ (since $u_1 \in {\cal L}$ throuand wghout the first part of the algorithm), hence $m_1$ has never been matched in any matching corresponding to the feasible basis visited by Scarf's algorithm. By Lemma~\ref{lem:u1-in-M}, in the next iteration we start with $u_1\in\cal{M}$. 

Now assume $u_1 \in {\cal M}$, and define $m_1$ to be the separator. We continue with a cardinal pivot to let some $m_1$-disliked edge enter our basis. Using Algorithm~\ref{alg:pivoting}, we can repeat arguments similar to Lemma \ref{lem:man-order}, Lemma \ref{lem:cardpivot}, Lemma~\ref{lem:wdisordpivot} and Lemma \ref{lem:mloopordpivot}, to conclude  the following. Details are given in Appendix~\ref{sec:app:proof-ordinpivot}.

\begin{lemma}\label{lem:m1lemma}
Consider an iteration of Scarf's algorithm that followed Algorithm~\ref{alg:pivoting} until $u_1 \in {\cal M}$. Let $D$ be the current almost-feasible ordinal basis assume we have $u_1\in \mathcal{M}$. Then 
\item[(i)] $1,2,\dots,k\notin D$.
\item[(ii)] The unique $m_1$-disliked edge in $D$ corresponds to $(m_1,w)$ for some $w\in W$.
\item[(iii)] In any cardinal pivot, suppose $B$ is our current feasible basis and $D=B\cup\{j_t\}\setminus\{1\}$ is the associated ordinal basis. If $e_{j_t}$ is a man-disliked valid edge, then we can find either the loop $e_1$ or some woman-disliked edge $e_{j_\ell}$ to leave $B$, and we obtain $B'=B\cup\{j_t\}\setminus\{j_\ell\}$ as a new feasible basis.
\item[(iv)] If the leaving edge $e_{j_\ell}$ is woman-disliked, then we do not terminate. In the following ordinal pivot, let $e_{j_r}$ be the reference edge. Then $v_{i_\ell}$ is a woman and $v_{i_r}$ is a man. A new man-disliked valid edge $e_{j^*}$ enters $D$. Moreover, denote $u'$ as the utility vector of the new ordinal basis $D'$, then
$$u_{i_\ell}'>u_{i_\ell}, \quad 
u_{i_r}'<u_{i_r}, \quad 
u_{i}'=u_{i}\textrm{ for $i\neq i_\ell,i_r$}.$$
In particular, $\sum_{w \in W}u_w$ strictly increases. \item[(v)] If the leaving edge is the loop $e_1$, then column $1$ leaves $B$, then we obtain $B'=D$, which terminates the algorithm.
\end{lemma}

These principles will indicate us to continue our anti-cycling pivots and end with (v) after $O(k^2)$ steps. Combining Lemma~\ref{lem:u1-in-M} with Lemma~\ref{lem:m1lemma} and bounding, in the case $u_1 \in {\cal M}$, the number of steps in a similar fashion as it was done when $u_1 \in {\cal L}$, we can bound the total running time of the algorithm.

\begin{theorem}\label{Thm:bipartitepolynomial}
If Scarf's algorithm iteratively applies Algorithm~\ref{alg:pivoting} to perform a cardinal pivot, then it converges after $O(n^2)$ steps. 

\end{theorem}

\section{Convergence Through a Perturbation of the Bipartite Matching Polytope}\label{sec:perturbation}
It is common in literature to run Scarf's algorithm preceded by a perturbation, so that the resulting polytope is non-degenerate~\cite{biro2016fractional,biro2016matching}. Recall that, in this case, the behavior of Scarf's algorithm is univocally defined, see Section~\ref{sec:scarf-lemma}. In this section, we give another perspective on our pivot rule connecting it to the perturbation approach.

Notice that the bipartite matching polytope given in form (\ref{eq:Ax-leq-b}) is highly degenerate. For $A,b$ as given in Section~\ref{sec:matrixdesign}, we define a non-degenerate polytope
\begin{equation}\label{eq:nondegeneratepoly}
    \{x\in\R^m_{\ge 0}:Ax=b+b(\epsilon)\},
\end{equation}
where $b(\epsilon)\in\Q^n_{\ge 0}$ is a parameterized vector defined by
$$b(\epsilon):=(\underbrace{\epsilon^{k+1},\epsilon^{k+2},\dots,\epsilon^{2k}}_{\hbox{men}},\underbrace{\epsilon,\epsilon^2,\dots,\epsilon^k}_{\hbox{women}})^T.$$

\begin{observation}\label{lemma:epsilon-woman-man}
In~\eqref{eq:nondegeneratepoly}, we have that $\sum_{m\in M}b_m(\epsilon)<b_w(\epsilon)$ for every woman $w$. In particular, the right-hand side of each constraint corresponding to a woman is strictly larger than the right-hand side of each constraint corresponding to a man.
\end{observation}

The following is a classical fact of linear algebra, see, e.g.,~\cite[Exercise 3.15]{bertsimas1997introduction}.

\begin{lemma}\label{obs:positivebasis}
There exists some $\epsilon^*>0$, such that for any fixed $\epsilon$ with $0<\epsilon<\epsilon^*$
\begin{itemize}
    \item[(i)]
All basic feasible solutions to the polytope defined in (\ref{eq:nondegeneratepoly}) are nondegenerate. That is, every basic feasible solution $x$ of (\ref{eq:nondegeneratepoly}) has exactly $n$ strictly positive entries. 
\item[(ii)] Every basis that is feasible for~\eqref{eq:nondegeneratepoly} is feasible for the original polytope.
\end{itemize} 
\end{lemma}

For any fixed $b(\epsilon)$ that makes the polytope nondegenerate, the iterations of Scarf's algorithm generate a sequence of Scarf pairs
\begin{equation}\label{eq:iterations}
    (B_0,D_0)\to(B_1,D_1)\to\dots\to(B_I,D_I).
\end{equation}
Recall that such that each pair contains two $n$-sets with $|B_I\cap D_I|=n-1$ for $I<N$ and $B_N=D_N$.\par
Our pivot rule (Algorithm~\ref{alg:pivoting}) can be captured by a specific perturbation:
\begin{theorem}\label{thm:perturbation}
There exists $\epsilon>0$ such that an execution of Scarf's algorithm on the original bipartite matching polytope with pivot rule given by Algorithm~\ref{alg:pivoting} gives the same sequence~\eqref{eq:iterations} as running it on its perturbation~\eqref{eq:nondegeneratepoly}. 
\end{theorem}
\begin{proof}

Fix $\epsilon=\min\{\frac{1}{2n+1},\epsilon^*\}$, where $\epsilon^*$ makes Lemma~\ref{obs:positivebasis} valid. 
Then (\ref{eq:nondegeneratepoly}) is nondegenerate.\par
Let $B$ be a feasible basis of~\eqref{eq:nondegeneratepoly}, corresponding to vertex $x^\epsilon$. By Lemma~\ref{obs:positivebasis}, $B$ is also feasible for the original bipartite matching polytope, and let $x$ be the basic feasible solution corresponding to it. If $j\in B$,
$$|(x^\epsilon-x)_j|
 =|(A_B^{-1}b(\epsilon))_j| =\left|\sum_{i=1}^n(A_B^{-1})_{ji}b_i(\epsilon)\right| \leq\sum_{i=1}^n|(A_B^{-1})_{ji}||b_i(\epsilon)| 
\leq\frac{1}{2n+1}\sum_{i=1}^n|(A_B^{-1})_{ji}|< \frac{1}{2},
$$where the last inequality uses the fact that $A$ is totally unimodular. \par
Combining the inequality above and the fact that $x_j=x^\epsilon_j=0$ for $j\notin B$, we have
\begin{claim}\label{cl:half}
For any $j\in[n]$, $x_j=1$ if and only if $x^\epsilon_j>\frac{1}{2}$.
\end{claim}
Notice that Lemma~\ref{lem:forest} is independent of the right-hand side, thus the graph representation $G_B$ of polytope (\ref{eq:nondegeneratepoly}) still has the forest with single loops structure. Now consider a tree $T$ with single loop from the forest. We define the root of $T$ as the node $r$ such that its loop belongs to $E_B$. Then $r$ is uniquely defined. 

We now investigate properties of the vector $x$ restricted to $T$. For any node $v$, recall that $\deg_T(v)$ is defined as the number of edges (including loops) incident to $v$ in $T$.  We call a node $v$ of $T$ a \emph{leaf} if $v\neq r$ and $\deg_T(v)=1$. $x_e$ for each edge $e$ incident to a leave of $T$ is uniquely defined, and we can then inductively define the other values of $x_e$ for each edge $e$ of $T$. 

Our perturbation imply the following property of leaves.
\begin{claim}\label{cl:leafisman}
If $v$ is a leaf of $T$, then $v$ is a man.
\end{claim}
\noindent \emph{\underline{Proof of Claim~\ref{cl:leafisman}}}. 
If a woman $w$ has degree $1$ in $T$ and $(w,w)\notin E_B$, then consider the only edge $(m,w)\in E_B$. We have that $x^\epsilon_{(m,w)}=1+b_w(\epsilon)\ge 1+\epsilon^{k}>1+b_m(\epsilon)$ by Observation~\ref{lemma:epsilon-woman-man}, which contradicts the feasibility of $B$ since we require that any edge in $E_B$ incident to $m$ has nonnegative $x^\epsilon$-value.\hfill $\small \blacksquare$\par

\begin{claim}\label{cl:womandeg2}
For any woman $w$, we have $\deg_T(w)=1$ or $\deg_T(w)=2$.
\end{claim}
\noindent \emph{\underline{Proof of Claim~\ref{cl:womandeg2}}}. 
If $T$ only contains $w$, then we have $\deg_T(w)=1$ since the loop $(w,w)$ is the only edge on $T$. For the rest of the proof, we assume $T$ has at least two nodes.\par
  So assume that $T$ contains at least two nodes. If $w$ is not a root, then from Claim~\ref{cl:leafisman}, we know $\deg_T(w)\neq 1$. Suppose $\deg_T(w)\ge 3$. Since~\eqref{eq:nondegeneratepoly} is integral when $b(\epsilon)$ is the $0$ vector, there exists at least two distinct men $m_a,m_b$ such that $e_j=(m_a,w)\in E_T$, $e_k=(m_b,w)\in E_T$, and $x_j=0$, $x_k=0$. For $m_a$, there exists another edge $e_{j'}$ incident to $m_a$ such that $x_{j'}=1$ (again, using the integrality of the polytope). The other endpoint of $e_{j'}=(m_a,w_{a})$ is a woman $w_{a}$ but cannot be a leaf by Claim~\ref{cl:leafisman}, which implies there exists at least one man $m_{a'}$ such that $e_{j''}=(m_{a'},w_{a})\in E_T$ with $x_{j''}=0$ (again using integrality). Continue extending this path, we have a sequence $(w,m_a,w_a,m_{a'},\dots)$ with every edge consisting of two consecutive nodes presenting in $T$. This sequence will not end unless it contains the root. Notice that the same argument holds for the sequence $(w,m_b,w_b,m_{b'},\dots)$ defined in the same way. Sincewhich is a one root in $T$ and $T$ has no cycles, we must have $m_a=m_b$, a contradiction.
  
  If $w$ is the root, then $\deg_T(w)\neq 1$ since $T$ is not a singleton. Suppose $\deg_T(w)\ge 3$. Then there exists a man $m_a$ such that $e_j=(m_a,w)\in E_T$ and $x_j=0$ (again by integrality). Similarly as before, we can continue this path with another edge $e_{j'}=(m_a,w_a)$ such that $x_{j'}=1$, and so on. We obtain a sequence $(w,m_a,w_a,\dots)$. Since $T$ is finite, we must end the sequence with a root, but since there is only one root $w$ on $T$ and $T$ has no cycles, we obtain a contradiction.\hfill $\small \blacksquare$\par
\medskip

Using Claim~\ref{cl:womandeg2} and Claim~\ref{cl:leafisman}, we conclude that the vector $x$ restricted to $T$ has a very regular structure, discussed in the next claim.

\begin{claim}\label{cl:xalternating}
For any tree $T$ from the forest with single loops structure of $G_B$, we have that:
\begin{enumerate}
    \item\label{it:cl:xalternating-1} Any path on $T$ starting from the root $r$ is $x$-alternating.
    \item\label{it:cl:xalternating-2} If the root is a man $m$, then the loop has $x_{(m,m)}=1$.
    \item\label{it:cl:xalternating-3} If the root is a woman $w$, then the loop has $x_{(w,w)}=1$ if $T$ only has one node $w$. Else if $T$ has at least two nodes, then $x_{(w,w)}=0$.
\end{enumerate}
\end{claim}
\noindent \emph{\underline{Proof of Claim~\ref{cl:xalternating}}}. 

1. If $T$ only contains one node, then by the definition in Section~\ref{sec:notation} it is clearly $x$-alternating.\par Suppose $T$ has at least two nodes. Fix a path $P$ in $T$ such that $r\in V_P$. Notice that by Claim~\ref{cl:womandeg2}, a woman cannot be incident to two edges $e,e'$ with $x_e=x_{e'}=0$ because of feasibility. Hence, if along the path $P$ we have two consecutive $0$'s, i.e., $e,e'\in E_P$, $e\cap e'\neq \empty$ and $x_e=x_{e'}=0$ then $e\cap e'=m$ for some man. By feasibility, we can find a third edge $\tilde{e}$ such that $m\in \tilde{e}$ and $x_{\tilde{e}}=1$. However, we can start from $\tilde{e}$ and find another $x$-alternating path $\tilde{P}$ starting at $m$ and edge-disjoint from $P$ because, by feasibility, we can always find a successor of a man with an incoming $0$-edge, and, by degree count, find a successor of a woman. Similar to the proof of Claim~\ref{cl:womandeg2}, the path $\tilde{P}$ must end at the root, which is a contradiction since $E_P\cap E_{\tilde{P}}=\emptyset$. This shows~\ref{it:cl:xalternating-1}.

\ref{it:cl:xalternating-2} and~\ref{it:cl:xalternating-3}. Consider the following parity argument: When $T$ is not a singleton, then there must exist a leaf, we $m$ by Claim~\ref{cl:womandeg2}. $m$ creates an $x$-alternating path starting from $e$ with $m\in e$ and $x_e=1$, which ends at the root. Then 2,3 hold. 
\hfill $\small \blacksquare$\par
\medskip

By Claim~\ref{cl:xalternating}, $T$ must be one of the following four types, presented in Figure~\ref{fig:perturbation-example}: 
\begin{enumerate}
    \item[(a)] A tree rooted at a man $m$, with at least two nodes. $x_{(m,m)}=1$ and every path from the root to a leaf is $x$-alternating.
    \item[(b)] A tree rooted at a woman $w$, with at least two nodes. $x_{(w,w)}=0$ and every path from the root to a leaf is $x$-alternating.
    \item[(c)] A singleton man $m$ with $x_{(m,m)}=1$.
    \item[(d)] A singleton woman $w$ with $x_{(w,w)}=1$.
\end{enumerate}

\begin{figure}
    \centering
    \includegraphics{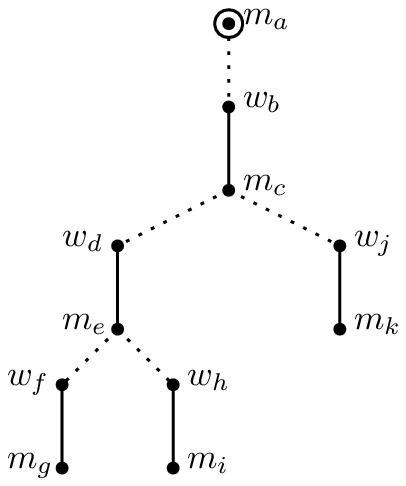}\hspace{1.5cm}
    \includegraphics{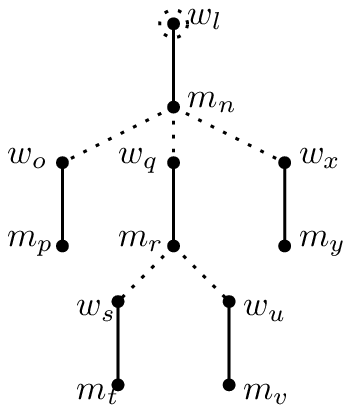}\hspace{1.5cm}
    \includegraphics{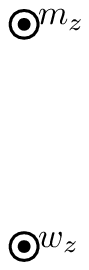}
    \caption{An illustration of four types of $T$. They are corresponding to type (a),(b),(c)(d) from left to right. The solid edges (resp.~dotted edges) are associated to $x$-value $1$ (resp.~$0$). Notice that every path from the root to the leaf is $x$-alternating.}
    \label{fig:perturbation-example}
\end{figure}

We now investigate $x^\epsilon$. Our key observation is that, for the $T$ of type (a) and type (b), the $x^\epsilon$-value has a monotonic property:
\begin{claim}\label{cl:xepsilon}
Let $T$ be a tree of $G_B$ of type (a) or type (b). Fix any path $P=r,f_0,r,f_1,v_1,f_2,\dots,f_p,v_p$ from the root $r$ to a leaf, where $f_0=(r,r)$ is the loop, $f_i=(v_{i-1},v_i)$ for $i\in [p]$, and $v_p$ is a man by Claim~\ref{cl:leafisman}. Then we have for $0\le i<j\le p$:
\begin{enumerate}
    \item\label{it:cl:xepsilon-0} If $x_{f_i}=x_{f_j}=0$, then $\frac{1}{2}\ge x^\epsilon_{f_i}>x^\epsilon_{f_j}>0$.
    \item\label{it:cl:xepsilon-1} If $x_{f_i}=x_{f_j}=1$, then $\frac{1}{2}<x^\epsilon_{f_i}<x^\epsilon_{f_j}$. Moreover, $x^\epsilon_{f_j}>1$ if and only if $j=p$.
\end{enumerate}
\end{claim}
\noindent \emph{\underline{Proof of Claim~\ref{cl:xepsilon}}}. 
Let us consider a simple case when $T$ itself is a path. Hence, $P=T$. Then we can directly obtain the value of $x^\epsilon$  by reversing the order in which nodes are visited in $P$. Indeed, by feasibility we have:
\begin{equation}
\begin{split}
    & x^\epsilon_{f_p}=1+b_{v_p}(\epsilon), \\
    & x^\epsilon_{f_{p-1}}=1+b_{v_{p-1}}(\epsilon)-x^{\epsilon}_{f_p}=b_{v_{p-1}}(\epsilon)-b_{v_{p}}(\epsilon),\\
    & x^\epsilon_{f_{p-2}}=1+b_{v_{p-2}}(\epsilon)-x^{\epsilon}_{f_{p-1}}=1+b_{v_{p-2}}(\epsilon)+b_{v_{p}}(\epsilon)-b_{v_{p-1}}(\epsilon),\\
    & ...\\
\end{split}
\nonumber
\end{equation}
We can use induction to obtain the formula
\begin{equation}\label{eq:xepsilon}
    x^\epsilon_{f_{p-k}}=\left\{
    \begin{array}{cc}
     1-\left(\sum_{\ell=1}^{k/2} b_{v_{p-k+2\ell-1}}(\epsilon)-\sum_{\ell=0}^{k/2} b_{v_{p-k+2\ell}}(\epsilon)\right),  & \textrm{if $k$ is even} \\
     \sum_{\ell=1}^{(k+1)/2} b_{v_{p-k+2\ell-1}}(\epsilon)-\sum_{\ell=0}^{(k-1)/2} b_{v_{p-k+2\ell}}(\epsilon),   & \textrm{if $k$ is odd}
    \end{array}.
    \right.
\end{equation}
Notice that $k$ is even iff $x_{f_{p-k}}=1$, also iff $v_{p-k}$ is a man. Formula~\eqref{eq:xepsilon} can be translated into a more intuitive form. For a generic tree $T$ (not necessarily a path), we say $v$ (resp.~$e$) is a node (resp.~edge) \emph{after} $v_k$ on $T$ if there is a path $P=r,f_0,r,f_1,v_1,f_2,\dots,f_p,v_p$ so that, when we remove edge $f_k$, then $v$ (resp.~$e$) and $v_k$ are are still connected. Notice that $v_k$ is always after $v_k$ itself. \par
Going back to the case when $T$ is a path, the following can be easily deduced from~\eqref{eq:xepsilon}. If $x_{f_k}=1$, then 
\begin{equation}\label{eq:xepsilonman}
    x^\epsilon_{f_k}=1+\sum_{\substack{\textrm{$m$ is a man after $v_k$} }} b_m(\epsilon)-\sum_{\textrm{$w$ is a woman after $v_k$}} b_w(\epsilon),
\end{equation}
else if $x_{f_k}=0$, then
\begin{equation}\label{eq:xepsilonwoman}
    x^\epsilon_{f_k}=\sum_{\substack{\textrm{$w$ is a woman after $v_k$}}} b_w(\epsilon)-\sum_{\textrm{$m$ is a man after $v_k$}} b_m(\epsilon).
\end{equation}
It is not hard to see~\eqref{eq:xepsilonman} and~\eqref{eq:xepsilonwoman} hold in general. In fact, we say that $T$ has a \emph{branching} (\emph{appearing at $v$}) if there is $v$ such that $\deg_T(v)\geq 3$. We show that the two formulas are still true. We first observe that by Claim~\ref{cl:womandeg2}, the branching can only appear at some man. Let $v_k$ be a man with $\deg_T(v_k)\ge 3$. Assume every edge after $v_k$ on $T$ makes~\eqref{eq:xepsilonman} and~\eqref{eq:xepsilonwoman} true. Denote $P^{(\alpha)}=r,f_0,r,f_1,v_1,f_2,\dots,f_k,v_k,f^{(\alpha)}_{k+1},v^{(\alpha)}_{k+1},\dots$ as the paths that pass through $v_k$ with distinct branches $f^{(\alpha)}_{k+1}$ for $\alpha=1,\dots,\deg_T(v_k)-1$. Then by feasibility, 
\begin{equation}
\begin{split}
    x^\epsilon_{f_k}&=1+b_{v_k}(\epsilon)-\sum_{\alpha=1}^{\deg_T(v_k)-1}x^\epsilon_{f^{(\alpha)}_{k+1}}\\
    &=1+b_{v_k}(\epsilon)-\sum_{\alpha=1}^{\deg_T(v_k)-1}\left(\sum_{\substack{\textrm{$w$ is a woman after $v^{(\alpha)}_{k+1}$}\\ \textrm{(including $v^{(\alpha)}_{k+1}$)}}} b_w(\epsilon)-\sum_{\textrm{$m$ is a man after $v^{(\alpha)}_{k+1}$}} b_m(\epsilon)\right)\\
    &=1+\sum_{\substack{\textrm{$m$ is a man after $v_k$}\\ \textrm{(including $v_k$)}}} b_m(\epsilon)-\sum_{\textrm{$w$ is a woman after $v_k$}} b_w(\epsilon).
\end{split}
\end{equation}
Thus~\eqref{eq:xepsilonman} holds for $v_k$. For~\eqref{eq:xepsilonwoman}, since $v_k$ is a woman, we have $\deg_T(v_k)\leq 2$. Hence, no branching appears at $v_k$, and we deduce that~\eqref{eq:xepsilonwoman} immediately by the fact that~\eqref{eq:xepsilonman} holds for $v_{k+1}$ and feasibility. By induction from the leaf to the root, we can verify 
that both~\eqref{eq:xepsilonman} and~\eqref{eq:xepsilonwoman} hold for any edge on $T$.\par
Notice that we always have $\sum_{m\in M}b_m(\epsilon)<b_w(\epsilon)$ for any $m\in M$, $w\in W$. Hence, the cumulative effects on~\eqref{eq:xepsilonman} and~\eqref{eq:xepsilonwoman} make the monotonicity property~\ref{it:cl:xepsilon-1} and~\ref{it:cl:xepsilon-0} true, respectively. In fact, if $x_{f_i}=x_{f_j}=0$ and $i<j$, then $f_i$ is closer to $r$ than $f_j$. By~\eqref{eq:xepsilonwoman}, there are more women after $v_i$ than those after $v_j$, which makes $x^\epsilon_{f_i}$ larger, since $\sum_{m\in M}b_m(\epsilon)<b_w(\epsilon)$ for any $w\in W$. $x^\epsilon_{f_i}\le \frac{1}{2}$ holds because of Claim~\ref{cl:half}. We can also obtain the second statement of Claim~\ref{cl:xepsilon} similarly. \hfill $\small \blacksquare$\par
\medskip

Now we have all the ingredients to conclude the proof of Theorem~\ref{thm:perturbation}. It suffices to show that, if $(B,D)$ is the current Scarf pair, then Scarf's algorithm both on the original bipartite matching polytope following Algorithm~\ref{alg:pivoting} and on the perturbed polytope~\eqref{eq:nondegeneratepoly} will move to the same new Scarf pair $(B',D').$ We will prove this assuming $u_1 \in {\cal L}$. Similarly to the discussion in Section~\ref{sec:convergence}, a similar argument settles the case $u_1 \in {\cal M}$.

By induction hypothesis, in both cases the variable entering $B$ is the unique variable contained in $D\setminus B$, and as usual we denote it by $j_t$. In the non-degenerate case, we increase $x^\epsilon_{j_t}$ from $0$ to some strictly positive value (because of non-degeneracy), until there is $j_\ell\in B$ such that $x^\epsilon_{j_\ell}=0$. Since~\eqref{eq:nondegeneratepoly} is non-degenerate, the choice of $j_\ell$ is unique.

From Lemma~\ref{lem:cardpivot}, when $j_t$ enters, there are two possible cases:
\begin{enumerate}
\item[(I)]$e_{j_t}$ joins two different trees $T_1,T_2$ of $(V,E_B^v)$, as to form a larger tree $T$ with two loops, or 
\item[(II)] $e_{j_t}$ connects two nodes of a same tree $T_1$ of $(V,E_B^v)$,
\end{enumerate}
where in both cases $V_{T_1}$ contains the separator $v_i$. Hence, $T_1$ is of type (a).

Suppose we are in case (I). Then by Proposition~\ref{prop:separator}, and using $j_t \in E_D$, $T_2$ cannot contain a loop at a man. Thus $T_2$ can only be of type (b) or type (d). Let $P$ be the path of form~\eqref{eq:alter-disliked-path}, connecting the two loops. By Claim~\ref{cl:P-starts-with-loop}, when $e_{j_t}$ is one of the man-disliked edges, then all edges which decrease their $x$-value are woman-disliked, except the loop $(m_i,m_i)$ which is disliked by man $m_i$. This also holds for $x^\epsilon$, since the sign of changes is independent from the $b$ vector. Moreover, all the decreasing variables will equally change\footnote{The cardinal pivot maintains feasibility in the sense of $A_Bx_B+A_{j_t}x_{j_t}=b+b(\epsilon)$, thus $x_B=A_B^{-1}(b+b(\epsilon))-A_B^{-1}A_{j_t}x_{j_t}$. Therefore, $x_{j_i}$ will change as $x_{j_t}$ changes with derivative $(A_B^{-1}A_{j_t})_{j_i}$. Since both $A_B^{-1}$ and $A_{j_t}$ are totally unimodular matrices, if the derivative is negative, it can only be $-1$, thus all decreasing variable will change with the same speed.}. Hence, we will select from all the decreasing variables the one that smalest value in $x^\epsilon$ leave the basis because this variable will be the first one reaching $0$.

If $T_2$ is of type (b), then we know that $x_{(w_{i_{\frac{p}{2}}},w_{i_{\frac{p}{2}}})}=0$ by Claim~\ref{cl:xalternating}, and then by Claim~\ref{cl:xepsilon} we have $x^\epsilon_{(w_{i_{\frac{p}{2}}},w_{i_{\frac{p}{2}}})}\le \frac{1}{2}$. Thus $e_{j_\ell}$ must be woman-disliked because we know that $x^\epsilon_{(m_i,m_i)}>\frac{1}{2}\ge x^\epsilon_{(w_{i_{\frac{p}{2}}},w_{i_{\frac{p}{2}}})}$, so the smallest entry cannot be $x^\epsilon_{(m_i,m_i)}$.\par
If $T_2$ is of type (d), then $x_{(w_{i_{\frac{p}{2}}},w_{i_{\frac{p}{2}}})}=1$ and by feasibility $x^\epsilon_{(w_{i_{\frac{p}{2}}},w_{i_{\frac{p}{2}}})}=1+b_{w_{i_{\frac{p}{2}}}}(\epsilon)>1$. The path $P$ is in fact $x$-augmenting. By the second statement in Claim~\ref{cl:xepsilon}, the smallest decreasing variable in $P$ is $x^\epsilon_{(m_i,m_i)}$, since $x^\epsilon_{(m_i,m_i)}<1<x^\epsilon_{(w_{i_{\frac{p}{2}}},w_{i_{\frac{p}{2}}})}$, and by monotonicity we have $x^\epsilon_{(m_i,m_i)}<x^\epsilon_e$ for any woman-disliked edge on $E_P\cap E_{T_1}$.\par
Therefore, if (I) happens, the leaving edge is selected following the rule of Algorithm~\ref{alg:pivoting}.\par
If (II) happens, when $e_{j_t}$ is man-disliked, then by Claim~\ref{cl:every-other-w-disliked}, the decreasing variables are all corresponding to woman-disliked edges, thus $e_{j_\ell}$ must be woman-disliked, which also follows the rule of Algorithm~\ref{alg:pivoting}.

\end{proof}

\section{On Expressing Stable Matchings as Dominating Vertices}\label{sec:not-all-stable}


\subsection{Failure in Representing the Intermediate Matchings}

It is natural to ask which stable matchings can be output using Scarf's algorithm (with any pivoting rule). Under the hypothesis of $C$ being consistent (see Section~\ref{sec:overview-expressing-sm}) we give a necessary condition for a stable matching to be represented by a dominating basis in the marriage model. Combined with Example~\ref{ex:Irving-leather}, our result shows that the idea of expressing the stable matchings as dominating vertices is limited, in the sense that in the worst case, the number of stable matchings is exponentially greater than the number of those corresponding to dominating vertices of $\mathcal{P}_M$. Recall that a stable matching $\mu$ is $v$-optimal (see Section~\ref{sec:overview-expressing-sm}) for some agent $v$ if $v$  is matched in $\mu$ to her best partner among all the partners she is matched to across all stable matchings.

\begin{definition}[Intermediate Stable Matching] \label{def:intermediatematching}
We call a stable matching $\mu$ \emph{intermediate}, if there is no $v\in V$ such that $\mu$ is $v$-optimal.
\end{definition}

As usual (see the discussion in Section~\ref{sec:scarf-lemma}), we restrict to complete instances with the same number of men and women.

\begin{theorem}\label{thm:intermediatematching}
Suppose we have a marriage instance $\mathcal{I}=(G(V,E),\succ)$ with complete preference lists and same number of men and women. Pick any consistent ordinal matrix $C$ (c.f.~Definition~\ref{def:C-consistency}). If a stable matching $\mu$ is intermediate, then there is no dominating vertex $x$ such that $x$ is the characteristic vector of $\mu$.
\end{theorem}

Theorem~\ref{thm:intermediatematching} tells us that under the consistency condition, there is no way to represent an intermediate stable matching by any dominating vertex. Thus, any implementation/pivoting rule of Scarf's algorithm will not lead us to such stable matchings. To prove this, we use the following well-known result, see~\cite{gusfield1989stable}.

\begin{lemma}\label{lem:one-side-optimal}
Consider a marriage instance ${\cal I}$ defined over a graph $G(M \cup W,E)$. Then, there exists two stable matchings $\mu_0$ and $\mu_z$ such that $\mu_0$ (resp.~$\mu_z$) is $m$-optimal (resp.~$w$-optimal) for any $m\in M$ (resp.~$w\in W$).
\end{lemma}

\begin{proof}[Proof of Theorem~\ref{thm:intermediatematching}.]

Fix a intermediate stable matching $\mu$ (we assume there exists one, otherwise the statement trivially holds). Fix $(A,b,C)$ such that $(A,b)$ defines the matching polytope of the instance and $C$ is consistent. Suppose by contradiction there exists a dominating vertex $x$ for $(A,b,C)$, such that $x$ is the characteristic vector of $\mu$, i.e.~$x_e=1$ if and only if $e\in \mu$.\par
Consider any dominating basis $B$ corresponding to the vertex $x$. Notice that if $e_j\in\mu$, then $j\in B$. Therefore, then,lity vector associated to the ordinal basis $B$, we have for any $i\in[n]$,
\begin{equation}\label{eq:utilityofdominatingbasis}
    u_i=\min_{j\in B} c_{ij}\le c_{i,(i,\mu(i))}.
\end{equation}

Notice that by Lemma~\ref{lem:forest}, there exists at least one loop in $E_B$. Suppose $(v_\ell,v_\ell)\in E_B$ for some $\ell$. Then, by consistency of $C$, we have 
\begin{equation}\label{eq:uell}
    u_\ell=\min_{j\in B} c_{\ell j}=c_{\ell,(v_\ell,v_\ell)}.
\end{equation}
Without loss of generality, we assume $v_\ell\in M$ is a woman.
Let $\mu_0$ be the man-optimal stable matching. Consider the edge $e=(m^*,v_\ell)$, where $m^*=\mu_0(v_\ell)$ is the partner of $v_\ell$ in $\mu_0$. Note that $m^*$ exists since we assume that the number of men and women coincide and lists are complete. By Lemma~\ref{lem:one-side-optimal}, from all stable matchings, $v_\ell$ is the best possible partner of $m^*$. Denote the column corresponding to $e$ as $c_e$, then, by consistency of $C$ and~\eqref{eq:uell},
\begin{equation}\label{eq:c>u-1}
    c_{v_\ell,e}>c_{v_\ell,(v_\ell,v_\ell)}=u_\ell.
\end{equation}
Notice that $v_\ell$ is the best possible partner among all partners $m^*$ is matched to in a stable matching, while $\mu(m^*)$ is less preferred by $m^*$ since $\mu$ is intermediate. Combining this observation with~\eqref{eq:utilityofdominatingbasis}, we obtain
\begin{equation}\label{eq:c>u-2}
    c_{m^*,e}>c_{m^*,(m^*,\mu(m^*))}\ge u_{m^*}.
\end{equation}
By the fact that a node $v \neq v_\ell,m^*$ is not incident to $e$ and the consistency of $C$, we have \begin{equation}\label{eq:c>u-3}
    c_{v,e}>u_v, \forall v\neq v_\ell,m^*.
\end{equation}
Now, by~\eqref{eq:c>u-1},~\eqref{eq:c>u-2}, and~\eqref{eq:c>u-3}, we find a column $c_e$ such that $c_e>u$, contradicting that $B$ is a dominating basis. Then the thesis follows.

\end{proof}

\subsection{An Example with Exponentially Many Intermediate Matchings}\label{sec:Irving-leather}

\begin{example}\label{ex:Irving-leather}
We give an infinite family of instances where the $v$-optimal stable matchings form an exponentially smaller subset of the set of all stable matchings.

For $n=2k \in 2 \mathbb{N}$, consider the instance with men $m_0,\dots, m_{k-1}$, women $w_0,\dots, w_{k-1}$. For $i \in \{0,\dots,k-1\}$, the preference list of man $m_i$ is given by:
$$
m_i : w_i \succ_{m_i} w_{i+1} \succ_{m_i} w_{i+\frac{k}{2}+1} \succ_{m_i} w_{i+\frac{k}{2}+2},
$$
where indices are taken modulo $k$. Women's lists also have length $4$ and are such that woman $w_j$ lists man $m_i$ in position $\ell$ if and only if man $m_i$ lists woman $w_j$ in position $5-\ell$. That is, for $i \in \{0,\dots,k-1\}$ we have:
$$
w_i : m_{i-\frac{k}{2}-2} \succ_{w_i} m_{i-\frac{k}{2}-1} \succ_{w_i} m_{i-1} \succ_{w_i} m_{i},
$$
where again indices are taken modulo $n$. See Table~\ref{table:stable-matching-instance} for an example.

\begin{table}[h!]
\centering
 \begin{tabular}{c||c c c c} 
 0 & 0 & 1 & 6 & 7 \\ 
1 & 1 & 2 & 7 & 8 \\ 
 2 & 2 & 3 & 8 & 9 \\
 3 & 3 & 4 & 9 & 0 \\
 4 & 4 & 5 & 0 & 1 \\
 5 & 5 & 6 & 1 & 2 \\
 6 & 6 & 7 & 2 & 3 \\
 7 & 7 & 8 & 3 & 4 \\
 8 & 8 & 9 & 4 & 5 \\
 9 & 9 & 0 & 5 & 6
 \end{tabular}
 \qquad \qquad 
  \begin{tabular}{c||c c c c} 
 0 & 3 & 4 & 9 & 0 \\ 
1 & 4 & 5 & 0 & 1 \\ 
 2 & 5 & 6 & 1 & 2 \\
 3 & 6 & 7 & 2 & 3 \\
 4 & 7 & 8 & 3 & 4 \\
 5 & 8 & 9 & 4 & 5 \\
 6 & 9 & 0 & 5 & 6 \\
 7 & 0 & 1 & 6 & 7 \\
 8 & 1 & 2 & 7 & 8 \\
 9 & 2 & 3 & 8 & 9
 \end{tabular}
 \caption{The instance constructed in Example~\ref{ex:Irving-leather} for $n=20$. On the left, preference lists or men are given, while on the right, preference lists of women are given.}\label{table:stable-matching-instance}
\end{table}
Lists are incomplete, but it is well-known that one can complete them by adding missing entries at the end of the preference lists, without changing the set of stable matchings, see, e.g.,~\cite{gusfield1989stable}. So the hypothesis from Theorem~\ref{thm:intermediatematching} hold, and we investigate the instance with incomplete lists for ease of exposition. For $i=0,\dots,k-1$, man $m_i$ lists woman $w_i$ as his top choice. Hence, the man-optimal stable matching $\mu_0$ assigns man $m_i$ to woman $w_i$. Similarly, the woman-optimal stable matching $\mu_z$ assigns each woman their favorite man. 

\end{example}

We next prove the required properties for Example~\ref{ex:Irving-leather}. The next claim shows that the man- and the woman-optimal stable matchings are the the only stable matchings that are $v$-optimal for some $v$.

\begin{claim}\label{cl:one-man-optimal}
Let $\mu$ be a $v$-optimal stable matching for some $v \in M \cup W$. Then $\mu\in \{\mu_0,\mu_z\}$.\end{claim}

On the other hand, the instance has exponentially many stable matchings.

\begin{claim}\label{cl:intermediate}
Let $S\subset  \{0,\dots,\frac{k}{2}-1\}$. Then the matching that, for $i \in S$, assigns men $i, i+\frac{k}{2}$ to their third-favorite partner and, for $i \in \{0,\dots,\frac{k}{2}-1\}\setminus S$, assigns men $i,i+\frac{k}{2}$  their second-favorite partner is stable.
\end{claim}

Let us show that Claim~\ref{cl:one-man-optimal} and Claim~\ref{cl:intermediate} imply the thesis. By Claim~\ref{cl:intermediate}, the instance has at least $\sqrt{2}^k$ stable matchings. On the other hand, by Claim~\ref{cl:one-man-optimal} there are exactly $2$ stable matchings that are $v$-optimal for some agent $v$.

The proofs of Claim~\ref{cl:one-man-optimal} and Claim~\ref{cl:intermediate} require the introduction of the classical concept of \emph{rotations exposed at a matching}, and is given in Appendix~\ref{app:exponentially-many}.

\subsection{When $C$ is not consistent}

foe in the next example, even if we drop the condition that $C$ is consistent, the broader class of all ordinal matrices $C$  does not help us represent all stable matchings, even if we allow for choices of $C$ such that some of the dominating vertices of $(A,b,C)$ are not stable matchings. We show this fact through the next example that is well-studied in many stable matching contexts, for instance in~\cite{dworczak2021deferred}:

\begin{example}\label{ex:scarfnegative}
Consider the following preference lists for $6$ agents: 

\begin{table}[h!]
\centering
 \begin{tabular}{c||c c c } 
 1 & 1 & 2 & 3  \\ 
 2 & 2 & 3 & 1  \\ 
 3 & 3 & 1 & 2  
 \end{tabular}
 \qquad \qquad 
  \begin{tabular}{c||c c c } 
 1 & 2 & 3 & 1  \\ 
 2 & 3 & 1 & 2  \\ 
 3 & 1 & 2 & 3  
 \end{tabular}
 \caption{The instance constructed in Example~\ref{ex:scarfnegative} for $n=6$. On the left, preference lists or men are given, while on the right, preference lists of women are given.}\label{table:6-agents-example}
\end{table}

We have three stable matchings: man-optimal $\mu_1=\{(m_1,w_1),(m_2,w_2),(m_3,w_3)\}$, intermediate $\mu_2=\{(m_1,w_2),(m_2,w_3),(m_3,w_1)\}$, and woman-optimal $\mu_3=\{(m_1,w_3),(m_2,w_1),(m_3,w_2)\}$.

We drop the condition of consistency on $C$ (c.f.~Section~\ref{sec:overview-expressing-sm}). We want to find an ordinal matrix $\tilde{C}$, such that all dominating vertices of $(A,b,\tilde{C})$ contain the three vertices $x_1,x_2,x_3$ of the polytope $\mathcal{P}_{FM}$ defined by $(A,b)$ such that $x_i$ is the characteristic vector of $\mu_i$ for $i=1,2,3$. Notice that, we do not even consider how to implement Scarf's algorithm to obtain them. Given the size of the instance, is not hard to do a complete enumeration of all matrices $\tilde C$ to conclude that such $\tilde{C}$ does not exist (recall that only the relative ordering of entries of $\tilde C$ matters). 
\end{example}

\section{Conclusions and Future Work}\label{sec:Conclusions}

Our paper shows what is, to the best of our knowledge, the first proof of polynomial-time convergence of Scarf's algorithm in relevant settings, as well as the first negative results on the expressive power of dominating vertices, hence of approaches that rely on Scarf's result. On one hand, we give supporting evidence that Scarf's algorithm can be proved to run in polynomial time in relevant cases, especially when we can leverage on a combinatorial interpretation of the input. On the other hand, we show that Scarf's algorithm can have structural limits much stronger than those of the search and enumeration problems it is associated to. 

Understanding Scarf's algorithm on the bipartite matching polytope foi ordinal matrices $C$ is an appealing theoretical question, although we are not aware of any model employing more general matrices $C$ on the bipartite matching polytope. Two specific questions here are in order. First, although we do not expect non-consistent $C$ to be meaningful for expressing stable matchings as dominating vertices, this statement requires a proof, and it would be useful to have an extension of Theorem~\ref{thm:Scarf-is-weak} to non-consistent matrices. Second, it would be interesting to understand whether Scarf's algorithm converges in polynomial time on the bipartite matching polytope for any ordinal matrix $C$. This would require an understanding of the problem that goes beyond stable marriages.


Future work also include the investigation of Scarf's algorithm in more general settings, as well as the relation between our pivoting rules and classical algorithms in the area, such as Tan's~\cite{tan1995generalization} and Roth-Vande Vate's random paths to stability~\cite{roth1990random}. 

\section*{Acknowledgments.} Yuri Faenza and Chengyue He acknowledge support from the National Science Foundation grant \emph{CAREER: An algorithmic theory of matching markets}.

\begin{appendix}

\section{Missing proofs}

\subsection{Proof of Lemma~\ref{lem:forest} and Lemma~\ref{lem:tree}}\label{sec:app:proof-basis}

\begin{proof}[Proof of Lemma~\ref{lem:forest}.]

First assume $B$ is a basis, i.e., $rank(B)=n$. To show the first claim, if $E_B^v$ contains a cycle, then it is an even cycle. An even cycle implies linear dependence of the corresponding columns in $B$, a contradiction.\par
Suppose there is a singleton node $v_i$, covered neither by a valid edge nor by its loop. Then the row $v_i$ of $B$ are all $0$'s, which makes $rank(B)\le n-1$, a contradiction.\par
Recall the following linear algebra fact: write $B=\left(\begin{array}{cc}B_1 & B_2 \\ B_3 & B_4\end{array}\right)$, and suppose that $B_4 = 0$ and $B_3$ has $q$ rows. Then $B_3$ has at least $q$ columns\footnote{This fact follows by subadditivity of the rank, since otherwise
$n=rk(B)\leq rk(B_1 \, , \, B_2) + rk(B_3) + rk(B_4) < n-q + q + 0 = n$, a contradiction.}. Consider a tree $T$ with at least two nodes (thus at least one valid edge), suppose first a tree of $E^v_B$ is incident to no loop. Then we can apply the statement above to $B$ by taking the rows of $B_3$ to be the node set of $T$, and its columns the edge set of $T$ and obtaining a contradiction. If conversely a tree of $E^v_B$ is incident to two or more loops, we can apply the statement above by taking the rows of $B_3$ to be all nodes except those of $T$, and by columns all columns of $B$ corresponding to edges not incident to nodes in $T$, obtaining again a contradiction.\par
\medskip
On the other hand, suppose the two graph conditions are satisfied. It is well-known that an incidence matrix of a tree can be permuted to have the form 
\begin{displaymath}
\begin{pNiceMatrix}[first-row,first-col]
     & e_{j_1} & e_{j_2} & \cdots & e_{j_t} \\
    v_{i_1} & 1 & * & & * \\
    v_{i_2} & 1 & * & & * \\
    v_{i_3} & 0 & 1 & & * \\
    \vdots & \vdots & \vdots & \ddots & \vdots \\
    v_{i_{t+1}} & 0 &0 & & 1
\end{pNiceMatrix}
\end{displaymath}
 with arbitrary node $v_{i_1}$ to be the first row (this fact can be shown by induction), and each column having exactly two $1$s. We can extend this fact to our tree with a loop $e_{j_{t+1}}$ incident, w.l.o.g., to $v_{i_1}$. After adding the loop column, the new $t\times t$ square submatrix has determinant $1$, implying full rank. We can apply this to every tree in our forest and obtain that $B$ has full rank, thus a basis.
\end{proof}

\begin{proof}[Proof of Lemma~\ref{lem:tree}.]

By Lemma~\ref{lem:forest}, suppose that the edges $E_B$ can be decomposed into $\tau$ connected components. Then for every $\omega\in [\tau]$, consider the $\omega$-th component, its local incidence matrix corresponds to a square submatrix $B_\omega$, up to permuting rows and columns. Thus we obtain the desired result.

 \end{proof}

\subsection{Proof of Lemma~\ref{lem:m1lemma}}\label{sec:app:proof-ordinpivot}
\begin{proof}

Let $(B,D)$ be the first Scarf pair with $u_1\in \mathcal{M}$ obtained by applying Algorithm~\ref{alg:pivoting}. 
We have already argued that, in the Scarf pair preceding $(B,D)$, we have $u_1 \in {\cal L}$. We can then apply Lemma~\ref{lem:mloopordpivot} part 2(ii) and deduce that 
(i),(ii) from Lemma~\ref{lem:m1lemma} are satisfied, and $e_{j_t}$ is a man-disliked valid edge.

Let now $(B,D)$ be a generic Scarf pair visited by the algorithm, and assume it verifies $u_1 \in {\cal M}$, (i), (ii) from Lemma~\ref{lem:m1lemma} and that $e_{j_t}$ is a man-disliked valid edge. We show that $(B,D)$ satisfies (iii) from Lemma~\ref{lem:m1lemma}. As a consequence of (iii), the algorithm either continues or terminates. If it continues, we obtain (iv) and that the next pair $(B',D')$ visited by the algorithm satisfies $u'_1 \in {\cal M} $, (i),(ii) from Lemma~\ref{lem:m1lemma} and $e_{j^*}$ (where $\{j^*\}=D'\setminus D$) is a man-disliked valid edge. Else, if we terminate, then we obtain $B'=D$ (i.e., $(v)$ is verified). This concludes the proof. 

\smallskip

We start with auxiliary claims. Recall that the basis structure (i.e., Lemma~\ref{lem:forest}) is independent of the ordinal basis and the utility vector. Consider the tree $T$ (with single loop, but we omit to repeat ``with single loop'' below) of $G_B$ containing $m_1$. Unlike the case $u_1 \in {\cal L}$, now $T$ is no longer a singleton, since we know from (ii) there is a valid edge $(m_1,w) \in E_D$.

\smallskip

\begin{claim}\label{cl:Tispivoting}
The entering edge $e_{j_t} (\notin E_B)$ is incident to $T$. In other words, $e_{j_t}$ and $T$ have at least one common node.
\end{claim}
\noindent \emph{\underline{Proof of Claim~\ref{cl:Tispivoting}}}. 
Since $e_{j_t}\notin E_B$, using Lemma~\ref{lem:forest}, $e_{j_t}$ will either connect two trees or form an even cycle in one tree. Both cases create a connected component $\Gamma$ in $E_{B\cup \{e_{j_t}\}}$ with the number of edges being one more than the number of nodes. Since $e_{j_t}\in E_{\Gamma}$, it suffices to show that $m_1$ is on $\Gamma$.\par
If $m_1$ is not on $\Gamma$, then all the edges on $\Gamma$ are contained in $E_D$ since $D=B\cup \{j_t\}\setminus\{1\}$. By Corollary~\ref{obs:almost-feasible-v-i}, any edge in $E_D$ is disliked in $D$ by either $m_1$ or one of its endpoints. Notice that on $\Gamma$ no edge can be disliked in $D$ by $m_1$ since we have the induction hypothesis (ii). Thus, any edge on $\Gamma$ is disliked in $D$ by one of the nodes also on $\Gamma$, which is impossible since the number of edges does not match the number of nodes on $\Gamma$.\hfill $\small \blacksquare$\par
\medskip

The following shows that Lemma~\ref{unmatchedwoman} holds even though we do not have $u_1\in\mathcal{L}$.
\smallskip

\begin{claim}\label{cl:unmatchedwoman}
Either $B'=D$ or there exists $\bar w \in W$ such that $(\bar{w},\bar{w}) \in E_{B'} \cap E_{D'}$ and $\bar w$ is not properly matched in $\mu_{B'}$, where $\mu_{B'}$ is the matching corresponding to the feasible basis $B'$.
\end{claim}
\noindent \emph{\underline{Proof of Claim~\ref{cl:unmatchedwoman}}}. 
Let $x,x'$ be the basic feasible solution associated with the basis $B,B'$, respectively. Assume $B'\neq D$. If $x'_{(m_1,m_1)}=1$, then we know the matching $\mu'_B$ does not properly match all agents. Thus, there exists some woman $\bar{w}$ who is also not properly matched. Hence, $x'_{(\bar{w},\bar{w})}=1$ and $(\bar{w},\bar{w})\in E_{B'}$. Since $B'\neq D$, we have $B'\setminus \{1\}\subseteq D'$, which implies $(\bar{w},\bar{w})\in E_{B'}\cap E_{D'}$. It therefore suffices to show that $x'_{(m_1,m_1)}=1$.

Now suppose that we are at the first iteration such that $u_1\in \cal{M}$ and $x'_{(m_1,m_1)}=0$. Then we have $x_{(m_1,m_1)}=1$ and $x'_{(m_1,m_1)}=0$, i.e.,~a non-degenerate cardinal pivot makes the value $x_{(m_1,m_1)}$ decrease. We will show that by Algorithm~\ref{alg:pivoting}, $x_{(m_1,m_1)}$ is a leaving variable, hence, by Lemma~\ref{lem:scarf-halts}, $B'=D$, a contradiction. 

Let $\Gamma$ be the connected component of $G_{B \cup \{j_t\}}$ that contains $e_{j_t}$. 

If $\Gamma$ contains an even cycle, then by Claim~\ref{cl:Tispivoting} $\Gamma$ contains exactly one loop $(m_1,m_1)$. Since the pivoting is non-degenerate, $x_{j_t}$ changes from $0$ to $1$. Since $e_{j_t}$ belongs to an even cycle, the cycle must be $x$-alternating, and the change of weights only happen inside the cycle. 
There is no way to change $x_{(m_1,m_1)}$ from $1$ to $0$ without violating the feasibility. Thus, $x'_{(m_1,m_1)}=0$ is impossible.\par
If, on the other hand, $e_{j_t}$ joins two different trees in $G_B$, then there are two loops on $\Gamma$. The change of weight appearing at $e_{j_t}$ will lead to the changes of weights on the $x$-augmenting path on $\Gamma$, which results in the change of weights on the loops. Therefore, $x'_{(m_1,m_1)}=0$, and by Algorithm~\ref{alg:pivoting}, $B'=D$.\par
Hence, once $x_{(m_1,m_1)}=1$ at some iteration, we have $x'_{(m_1,m_1)}=1$ at all future iterations until termination. Since $1$ is therefore not matched, there must be a woman $\bar w$ that is not matched, concluding the proof. 
\hfill $\small \blacksquare$\par

 

\medskip

To show the correctness of (iii), we can apply the above results and repeat the same arguments in Lemma~\ref{lem:cardpivot} with two subtle modifications as follows:\par
First, for formula (\ref{eq:alter-disliked-path}) in the proof:
$$P=(m_{i_1},m_{i_1}),(m_{i_1},w_{i_1}),(w_{i_1},m_{i_2}),\dots,(m_{i_{\frac{p}{2}}},w_{i_{\frac{p}{2}}}),(w_{i_{\frac{p}{2}}},w_{i_{\frac{p}{2}}}).$$
We assign $m_{i_1}=m_1$, and $w_{i_{\frac{p}{2}}}=\bar{w}$. Then the edges of $P$ are disliked by 
$$none,m_1,w_{i_1},m_{i_2},w_{i_2},\dots,m_{i_{\frac{p}{2}-1}},w_{i_{\frac{p}{2}-1}},m_{i_{\frac{p}{2}}},w_{i_{\frac{p}{2}}}$$
in order (notice that, $(m_{i_1},m_{i_1})=(m_1,m_1)$ does not exist in $E_D$ as $1\notin D$, thus is denoted by ``none''-disliked).\par
Second, for formula (\ref{eq:alter-disliked-cycle}) in the cycle case when $Q$ contains $m_1$:
$$Q=(m_{i_1},w_{i_1}),(w_{i_1},m_{i_2}),\dots,(m_{i_{\frac{p}{2}}},w_{i_{\frac{p}{2}}}),(w_{i_{\frac{p}{2}}},m_{i_1}).$$
Here $m_{i_1}=m_1$ and $w_{i_1}=w$ where $w$ is the woman defined in Lemma~\ref{lem:m1lemma}(ii). The edges on $Q$ are still disliked by $m_1,w_{i_1},m_{i_2},w_{i_2},\dots,m_{i_{\frac{p}{2}-1}},w_{i_{\frac{p}{2}-1}},m_{i_{\frac{p}{2}}},w_{i_{\frac{p}{2}}}$
in this order.\par
Use the remaining arguments in the proof of Lemma~\ref{lem:cardpivot}, we can deduce (iii).\par
To prove (iv), suppose $e_{j_\ell}$ is woman-disliked w.r.t.~$D$. Then $v_{i_\ell}$ is a woman with $c_{i_\ell,j_\ell}<c_{i_\ell,j_r}\in\mathcal{M}$. Thus $e_{j_r}$ is a valid edge and the other endpoint corresponds to a man, which is $v_{i_r}$. Then the utility of $v_{i_\ell}$ increases and that of $v_{i_r}$ decreases.\par
We still need to show that $e_{j^*}$ is a man-disliked valid edge w.r.t.~$D'$. It is man-disliked since $v_{i_r}$ is a man. Now suppose $e_{j^*}$ is a loop, then $e_{j^*}=(v_{i_r},v_{i_r})$. If $v_{i_r}= m_1$, we will terminate the algorithm at $(B',D')$ since $1\in D'$. However, as the loop $(m_1,m_1)$ does not leave $E_B$, we know $x'_{(m_1,m_1)}=1$ by our pivoting rule, which implies that $m_1$ is not properly matched by $\mu_{B'}$, thus $\mu_{B'}$ is not a stable matching, a contradiction with Theorem~\ref{thm:existence-sm}. Thus $v_{i_r}\neq m_1$. Since $v_{i_r}$ is a man, $(v_{i_r},\bar{w})$ is a valid edge. If $(v_{i_r},\bar{w})\notin E_{D'}$, then the corresponding column $c_{(v_{i_r},\bar{w})}$ is strictly greater than $u'$. Else if $(v_{i_r},\bar{w})\in E_{D'}$, then this edge cannot be disliked by anyone in $D'$, since $v_{i_r}$ and $\bar{w}$ dislike their loops, and $m_1$ dislikes some edge incident to $m_1$. Both cases contradict with the definition of the dominating basis.\par
Since the entering edge $e_{j^*}$ can never be a loop, we will maintain (i) for $(B',D')$. Also, we know that $u_1$ never increases by (iv), thus $v_{i_\ell}$ is a woman. Therefore (ii) also holds for $(B',D')$.\par
When $1$ leaves $B$, the algorithm terminates. If we are at (v), then we stop the induction.\end{proof}

\subsection{Missing proofs from Section~\ref{sec:not-all-stable}}\label{app:exponentially-many}

In order to prove Claim~\ref{cl:one-man-optimal} and Claim~\ref{cl:intermediate}, let us introduce the classical concept of rotations and related properties. For an extensive treatment of rotations, as well as proofs of basic facts on rotations stated here, we refer to~\cite{gusfield1989stable}. Given a stable matching $\mu$, a \emph{$\mu$-alternating cycle} is a cycle (in the classical sense) whose edges alternatively belong to $\mu$ and to $E\setminus \mu$. Let 
\begin{equation}\label{eq:rho} \rho = m_{i_0},e_0,w_{i_0'},e_0',m_{i_1},e_1,w_{i_1'},\dots, m_{i_{\ell-1}},e_{\ell}, w_{i_{\ell-1}'}\end{equation}
be a $\mu$-alternating cycle. We say that $\rho$ is a \emph{rotation exposed at $\mu$} if, for $j=0,\dots,\ell-1$,
\begin{enumerate}
    \item $m_{i_j}$ is matched to $w_{i_j}$ in $\mu$, and 
    \item $w_{i_{j+1}}$ is the first woman according to $m_{i_{j}}$'s preference list who prefers $m_{i_{j}}$ to her partner in $\mu$,  
\end{enumerate}
where indices are taken modulo $\ell$.

Rotations allow us to move from a stable matching to the other. More formally, by defining the symmetric difference operator $\Delta$, the following holds.

\begin{lemma}\label{lem:rotations-move}
Let $\rho$ be a rotation exposed at some stable matching $\mu$. Then $\mu':=\mu \Delta E_\rho$ is a stable matching. Moreover, for each man $m \in M$, either $m$ is matched to the same woman in $\mu$ and $\mu'$, or he prefers his assigned partner in $\mu$ to his assigned partner in $\mu'$. 
\end{lemma}

\begin{lemma}\label{lem:rotations-all}
Let $\mu$ be a stable matching. Then there is a sequence of stable matchings $\mu_0, \mu_1, \dots, \mu_s$ such that:
\begin{enumerate}
    \item $\mu_0$ is the man-optimal stable matching, while $\mu_s=\mu$;
    \item for $i=1,\dots, s$, there exists a rotation $\rho_i$ exposed at $\mu_{i-1}$ such that $\mu_i = \mu_{i-1} \Delta E_{\rho_i}$.
\end{enumerate}
\end{lemma}

Let us now prove an intermediate fact.

\begin{claim}\label{cl:example-intermediate}
The only rotation exposed at $\mu_0$ is 
$$\rho_0 = m_0, w_0,m_1,w_1,\dots, m_{k-1},w_{k-1}.$$
Moreover, $\mu':=\mu\Delta E_{\rho_0}$ is a stable matching where all men are matched to the second woman in their list. 
\end{claim}

\noindent \emph{\underline{Proof of Claim~\ref{cl:example-intermediate}}}. Since the instance has the stable matching $\mu_z\neq \mu_0$, there must be at least one rotation $\rho$ as in~\eqref{eq:rho} exposed in $\mu$ by Lemma~\ref{lem:rotations-all}. Assume w.l.o.g.~that $i_0=i$ for some $i =0,\dots,k-1$. Then, by definition of rotation, ${i_0'}=i$ and $i_1'=i+1$. Consequently, $i_1=i+1$. Iterating we deduce, $\rho=\rho_0$. Moreover, all men and all women are contained in $\rho$. Since all rotations exposed at a stable matching are vertex-disjoint (see again~\cite{gusfield1989stable}), $\rho_0$ is the only rotation exposed at $\mu_0$. 

Stability of $\mu'$ follows from Lemma~\ref{lem:rotations-move} and the fact that each man is matched to the second woman in his list by the definition of symmetric difference. \hfill $\small \blacksquare$\par
\medskip

We are now ready to prove Claim~\ref{cl:one-man-optimal} and Claim~\ref{cl:intermediate}. 

\medskip

\noindent \emph{\underline{Proof of Claim~\ref{cl:one-man-optimal}}}.  By Lemma~\ref{lem:rotations-all}, all stable matchings can be obtained from $\mu_0$ by iteratively taking the symmetric difference with rotations from a sequence, and each partial sequence of symmetric differences creates a  stable matching. By Claim~\ref{cl:example-intermediate}, the only rotation exposed in $\mu_0$ is $\rho_0$. Hence, all stable matchings other than $\mu_0$ can be obtained from $\mu'$ by a sequence of rotations eliminations. By Lemma~\ref{lem:rotations-move}, no man improves their partner when the symmetric difference with an exposed rotation is taken. Hence, in no stable matching other than $\mu_0$ a man has a partner he prefers to his partner in $\mu'$. Using again Claim~\ref{cl:example-intermediate}, we know that $\mu'$ is not $v$-optimal for any $v$. We conclude that $\mu_0$ is the only $m$-optimal matching for any man $m$.

To conclude the proof of the claim, it suffices to show that $\mu_z$ is the only $w$-optimal stable matching for any woman $w$. This follows by observing that $$m_i \leftrightarrow w_{i+\frac{k}{2}+2} \hbox{ for $i=0,\dots, k-1$}$$
defines a bijection between our class of instances and the instances where role of men and women as swapped (as usual, indices are taken modulo $k$). That is, the preference lists of the new instance are, for $i \in \{0,\dots, k-1\}$:
$$
w_i : m_i \succ m_{i+1} \succ m_{i+\frac{k}{2}} \succ m_{i+\frac{k}{2}+1} \quad \hbox{ and } \quad m_i : w_{i-\frac{k}{2}-2} \succ w_{i-\frac{k}{2}-1} \succ w_{i-1} \succ w_{i}.$$
The thesis then follows by the first part of the proof. \hfill $\small \blacksquare$\par
\medskip

\noindent \emph{\underline{Proof of Claim~\ref{cl:intermediate}}}. Consider again matching $\mu'$ defined in Claim~\ref{cl:example-intermediate}. For $i=0,\dots, \frac{k}{2}-1$, we have that
$$\rho_i:= m_{i}, w_{{i+1}}, m_{i+\frac{k}{2}}, w_{i+\frac{k   }{2}+1}
$$ is a rotation exposed in $\mu'$. It is well-known (see again~\cite{gusfield1989stable}) that if $\rho,\rho'$ are rotations exposed in a stable matching $\mu$, then $\rho'$ is a rotation exposed in $\mu \Delta \rho$. Hence, for each $S$ as in the hypothesis of the claim, we have that $(((\mu'\Delta \rho_{i_1}) \Delta \rho_{i_2})\dots \Delta \rho_{i_{|S|}})$ is a stable matching, where $S=\{\rho_{i_1},\rho_{i_2}, \dots, \rho_{i_{|S|}}\}$. Clearly, all such matchings are distinct. The statement follows. \hfill $\small \blacksquare$\par



\end{appendix}




\end{document}